%% file: main.tex
\documentclass{article}

\PassOptionsToPackage{numbers, compress}{natbib}

\input{preamble}

\title{Finding Simple Proofs for First-Order Optimization}
\author{%
  Daniel Berg Thomsen\thanks{corresponding author:
  \href{mailto:daniel.berg-thomsen@inria.fr}{daniel.berg-thomsen@inria.fr}}
  ~ \thanks{INRIA \& D.I. {\'E}cole Normale Sup{\'e}rieure, CNRS \& PSL
  Research University, Paris, France.} ~ \thanks{CMAP, {\'E}cole
  Polytechnique, Institut Polytechnique de Paris, Paris, France.} \\
  \and
  Manu Upadhyaya\footnotemark[2] ~ \footnotemark[3] \\
  \and
  Baptiste Goujaud\thanks{SAMOVAR, T{\'e}l{\'e}com SudParis, Institut
  Polytechnique de Paris, Palaiseau, France.} \\
  \and
  Aymeric Dieuleveut\footnotemark[3] \\
  \and
  Adrien Taylor\footnotemark[2]
}

\begin{document}

\maketitle

\input{sections/abstract}

\input{sections/introduction}

\input{sections/simple_proof_structures}

\input{sections/heuristic_search}

\input{sections/experimental_examples}

\input{sections/proximal_certificates}

\input{sections/conclusion}

\begin{ack}
  D.~Berg Thomsen, M.~Upadhyaya, and A.~Taylor are supported by the European
  Union (ERC grant
  CASPER 101162889). The work of A.~Dieuleveut is partly supported by
  ANR-19-CHIA-0002-01/chaire SCAI, Hi!Paris FLAG project, and PEPR Redeem.
  B.~Goujaud is supported by a Hi!Paris chair. The
  French government also partly funded this work under the management of the
  Agence Nationale de la Recherche as part of the France 2030
  program, references
  ANR-23-IACL-0008 ``PR[AI]RIE-PSAI'', ANR-23-PEIA-005 (REDEEM project), and
  ANR-23-IACL-0005.
\end{ack}

\bibliographystyle{abbrvnat}
\bibliography{bib}

\newpage

\appendix

\appendixtableofcontents

\input{sections/appendix}

\end{document}

%% file: preamble.tex
\usepackage[preprint]{neurips_2026}
\usepackage[utf8]{inputenc}
\usepackage[T1]{fontenc}
\usepackage{hyperref}
\usepackage{xurl}
\usepackage{booktabs}
\usepackage{amsmath,amssymb,amsfonts,amsthm,mathtools}
\usepackage{thmtools}
\usepackage{stmaryrd}
\SetSymbolFont{stmry}{bold}{U}{stmry}{m}{n}
\usepackage{algorithm}
\usepackage{algpseudocodex}
\usepackage{graphicx}
\graphicspath{{figures/}}
\usepackage{nicefrac}
\usepackage{microtype}
\usepackage{xcolor}
\usepackage{bm}
\usepackage{enumitem}
\usepackage[nameinlink,noabbrev]{cleveref}
\usepackage{cancel}

\hypersetup{hypertexnames=false, colorlinks=true, linkcolor=blue, citecolor=blue, urlcolor=blue}
\providecommand{\doi}[1]{}
\renewcommand{\doi}[1]{\href{https://doi.org/#1}{doi: #1}}
\input{commands}

\newtheorem{theorem}{Theorem}
\newtheorem{lemma}{Lemma}

\theoremstyle{definition}
\newtheorem{assumption}{Assumption}

\theoremstyle{remark}
\newtheorem{remark}{Remark}

\AddToHook{env/theorem/begin}{\crefalias{section}{theorem}}
\AddToHook{env/lemma/begin}{\crefalias{section}{lemma}}
\AddToHook{env/corollary/begin}{\crefalias{section}{corollary}}
\AddToHook{env/example/begin}{\crefalias{section}{example}}
\AddToHook{env/fact/begin}{\crefalias{section}{fact}}
\AddToHook{env/assumption/begin}{\crefalias{section}{assumption}}
\AddToHook{env/definition/begin}{\crefalias{section}{definition}}
\AddToHook{env/remark/begin}{\crefalias{section}{remark}}

\crefname{assumption}{assumption}{assumptions}
\Crefname{assumption}{Assumption}{Assumptions}
\crefname{appendix}{appendix}{appendices}
\Crefname{appendix}{Appendix}{Appendices}
\crefname{subappendix}{appendix}{appendices}
\Crefname{subappendix}{Appendix}{Appendices}
\AddToHook{cmd/appendix/before}{%
    \crefalias{section}{appendix}%
    \crefalias{subsection}{subappendix}%
}

\makeatletter
\newif\ifpaper@appendix
\AddToHook{cmd/appendix/after}{\paper@appendixtrue}

\let\paper@orig@section\section
\let\paper@orig@subsection\subsection

\RenewDocumentCommand{\section}{s o m}{%
    \IfBooleanTF{#1}{%
        \paper@orig@section*{#3}%
    }{%
        \ifpaper@appendix
            \clearpage
        \fi
        \IfNoValueTF{#2}{%
            \paper@orig@section{#3}%
        }{%
            \paper@orig@section[#2]{#3}%
        }%
        \ifpaper@appendix
            \IfNoValueTF{#2}{%
                \addcontentsline{atoc}{section}{\protect\numberline{\thesection}#3}%
            }{%
                \addcontentsline{atoc}{section}{\protect\numberline{\thesection}#2}%
            }%
        \fi
    }%
}

\RenewDocumentCommand{\subsection}{s o m}{%
    \IfBooleanTF{#1}{%
        \paper@orig@subsection*{#3}%
    }{%
        \IfNoValueTF{#2}{%
            \paper@orig@subsection{#3}%
        }{%
            \paper@orig@subsection[#2]{#3}%
        }%
        \ifpaper@appendix
            \IfNoValueTF{#2}{%
                \addcontentsline{atoc}{subsection}{\protect\numberline{\thesubsection}#3}%
            }{%
                \addcontentsline{atoc}{subsection}{\protect\numberline{\thesubsection}#2}%
            }%
        \fi
    }%
}

\newcommand{\appendixtableofcontents}{%
    \section*{Appendices}%
    \begingroup
        \hypersetup{linkcolor=black}%
        \setcounter{tocdepth}{2}%
        \noindent\textbf{Table of Contents}\par
        \vspace{0.25\baselineskip}%
        \hrule
        \vspace{0.35\baselineskip}%
        \@starttoc{atoc}%
    \endgroup
    \clearpage
}
\makeatother

%% file: commands.tex
\newcommand\Bigset[1]{\mathord{\left\{ #1 \right\}}}

\newcommand\Bigp[1]{\mathord{\left( #1 \right)}}
\newcommand{\abs}[1]{\left\lvert#1\right\rvert}
\newcommand{\Bignorm}[1]{\left\lVert#1\right\rVert}

\newcommand{\inner}[2]{\langle #1, #2 \rangle}

\newcommand{\A}{\mathcal{A}}
\newcommand{\KK}{\mathcal{K}}

\newcommand{\F}{\mathcal{F}}
\newcommand{\Pc}{\mathcal{P}}
\newcommand{\Ic}{\mathcal{I}}
\newcommand{\R}{\mathbb{R}}

\newcommand{\HH}{\mathcal{H}}

\newcommand{\labels}{\mathcal{S}}
\newcommand{\hor}{N}

\DeclareMathOperator*{\minimize}{minimize}
\DeclareMathOperator*{\argmin}{argmin}

\DeclareMathOperator{\tr}{Tr}

\renewcommand{\leq}{\leqslant}
\renewcommand{\geq}{\geqslant}
\renewcommand{\succeq}{\succcurlyeq}

\newlength{\problemheadwidth}
\newlength{\problemsubjecttowidth}
\newcommand{\problemhead}[2]{%
    \settowidth{\problemheadwidth}{$\displaystyle\underset{#2}{\text{#1}}$}%
    \settowidth{\problemsubjecttowidth}{$\displaystyle\text{subject to}$}%
    \ifdim\problemheadwidth<\problemsubjecttowidth
        \setlength{\problemheadwidth}{\problemsubjecttowidth}%
    \fi
    \global\problemheadwidth=\problemheadwidth
    \mathmakebox[\problemheadwidth][c]{\underset{#2}{\text{#1}}}\quad
}
\newcommand{\problemfind}{%
    \settowidth{\problemheadwidth}{$\displaystyle\text{find}$}%
    \settowidth{\problemsubjecttowidth}{$\displaystyle\text{subject to}$}%
    \ifdim\problemheadwidth<\problemsubjecttowidth
        \setlength{\problemheadwidth}{\problemsubjecttowidth}%
    \fi
    \global\problemheadwidth=\problemheadwidth
    \mathmakebox[\problemheadwidth][c]{\text{find}}\quad
}
\newcommand{\problemsubjectto}{%
    \mathmakebox[\problemheadwidth][c]{\text{subject to}}\quad
}

%% file: sections/abstract.tex
\begin{abstract}
Progress in mathematics often requires more than a certificate of truth: it
requires proof structures that are transparent, checkable, and reusable.
Automated systems can increasingly certify that a result is true; what they
typically return, however, is a dense certificate rather than an interpretable,
reusable proof structure.
    
Recent work on performance estimation problems has shown that performance bounds and complexity analyses of first-order optimization methods can be discovered by searching over a structured space of Lagrangian dual certificates. We cast the search for simpler proof structures as a second-stage optimization problem over these certificates. Starting from dual certificates, we develop post-processing procedures using tools from sparse optimization and statistical learning. We measure complexity through features such as active hypotheses and residual structure, and introduce methods based on exhaustive sparsification, weighted $\ell_1$-type heuristics, and semidefinite programming (SDP) formulations for discovering simple proofs and intermediate lemmas.

Examples on gradient descent, proximal methods, and fast-gradient methods show that these procedures can autonomously prune redundant inequalities, reveal structured proof patterns, and, in the proximal setting, recover Lyapunov functions as intermediate lemmas that lead to simple, streamlined proofs. By distilling dense machine-generated certificates into compact proof structures, this workflow acts as a pre-processing step for the final proof, reducing the complexity that must be managed during human interpretation, reuse, and formalization.
\end{abstract}

%% file: sections/introduction.tex
\section{Introduction}\label{sec:introduction}

Proof simplification is a central part of mathematical practice. Once a result
has been proved, a simpler proof can change what the result offers: it can make arguments easier to verify, reveal mechanisms behind the statement, and
turn an isolated derivation into a building block for extensions. This role becomes
especially important in computer-aided mathematics, where proofs and
certificates may be produced, manipulated, or checked by numerical solvers,
computer algebra systems, formal proof assistants such as Lean, and large
language models. In
such settings, obtaining a valid certificate is often not enough: one also wants
proof structures that are sparse, modular, and reusable.

In the context of performance bounds and complexity analyses of (first-order) optimization algorithms, performance estimation problems (PEPs) provide a concrete framework for formulating such simplification tasks precisely. PEPs have made it possible to certify tight
worst-case bounds for first-order optimization methods by solving numerical
optimization problems~\cite{drori2014performance,taylor2017smooth}. In
standard semidefinite programming (SDP) formulations of interpolation-based PEPs, Lagrangian dual solutions are algebraic certificates
of the corresponding worst-case performance bounds,
and hence machine-searchable proof objects. The certificates returned by
numerical solvers, however, are typically parameter-specific, dense, and far from
unique: even after the
method, function class, PEP formulation, and target bound have been fixed, many
different certificates may prove the same inequality. In this form, a certificate
may establish the rate without revealing how its inequalities, residual terms,
and intermediate statements can be organized into a simple proof.

This work asks whether the same machinery that certifies a worst-case
guarantee can also help identify a simple proof of it. We treat simplification
as a second-stage optimization problem over Lagrangian dual certificates~\citep{goujaud2023fundamental}. Here, we define
\emph{simplicity} operationally: a proof may use fewer active hypotheses,
organize the remaining nonnegative residual terms more transparently, expose
recognizable multiplier patterns, or isolate an intermediate lemma that turns a
dense identity into a reusable proof step. We turn this formulation into
certificate-simplification procedures based on exhaustive sparsification,
weighted $\ell_1$-type heuristics, and semidefinite programs for finding
candidate intermediate lemmas. Rather than prescribing a single proof template,
the workflow exposes proof patterns and candidate lemmas by searching over
certificate representations; after a bound has been certified, it distills the
certificate into proof ingredients suitable for human interpretation, reuse, and
translation into symbolic/formal proof systems~\citep[e.g.,][]{naldi2025solving}.

Concretely, we consider first-order methods for problems of the form
\begin{equation*}
    \minimize_{x \in \R^d}\; f\Bigp{x},
\end{equation*}
where $f:\R^d\to\R$ belongs to a function class $\F$; the same framework can also accommodate constraints and composite terms. Let
$\mathcal{T}_{\A}\Bigp{f,\hor,x_0}$ denote the set of all possible first $\hor+1$ iterates generated by a method $\mathcal{A}$ (e.g., gradient descent) applied on a function $f\in \F$, and initialized at~$x_0$.
Moreover, let $\Pc_f$ denote the performance measure of interest (e.g., $\|x_\hor-x_\star\|^2$), and $\Ic_f$ denote the initialization measure (e.g., $\|x_0-x_\star\|^2$). 
The associated worst-case problem corresponds to finding a dimension $d$, a function $f$ with optimal point $x_\star$, an initialization $x_0$, and a trajectory $(x_k)_{k\in\llbracket 0, N\rrbracket}$, maximizing the effective rate $\rho$ for those measures, i.e.,
\begin{equation*}
    \begin{aligned}
    \problemhead{maximize}{\substack{\rho\in\R, f \in \F,\ d \geq 1,\ x_0\in\R^d \\
        x_\star \in \argmin_{x \in \R^d} f\Bigp{x} \\
        \Bigp{x_k}_{k\in\llbracket 0,\hor \rrbracket} \in \mathcal{T}_{\A}\Bigp{f,\hor,x_0}}}
        &
        \rho \\
    \problemsubjectto
        & \Pc_f\Bigp{\Bigp{x_k}_{k\in\llbracket 0,\hor \rrbracket}, x_\star} = \rho \Ic_f\Bigp{\Bigp{x_k}_{k\in\llbracket 0,\hor \rrbracket}, x_\star}.
    \end{aligned}
    \tag{PEP}
    \label{eq:intro-pep}
\end{equation*}
This problem is infinite-dimensional because it optimizes over the function $f$
itself, not only over a finite trajectory. When an exact finite-dimensional SDP
reformulation is available~\citep{taylor2017smooth,taylor2017exact},
interpolation theory replaces the function variable
and trajectory admissibility conditions by finitely many sampled oracle values,
method constraints, and auxiliary variables; we then work with the resulting
finite problem and its Lagrangian dual certificates.

For an exact SDP reformulation, any dual feasible point with objective value at
most the target bound provides a numerical certificate of that bound. These
certificates are typically not unique:
the same guarantee can be witnessed by different choices of Lagrangian dual
multipliers and slack variables. Thus, once the PEP formulation and target
guarantee are fixed, simplification can be viewed as a second-stage search for
an alternative certificate that still proves the chosen bound.
Such simplified certificates can turn dense numerical multipliers into explicit
patterns and, in favorable cases, reveal proof templates that can be adapted to
related function classes or algorithms.

\paragraph{Contributions.}
\textbf{(i)} Building on standard SDP formulations for interpolation-based PEPs,
we introduce certificate-complexity criteria that count active inequalities and
active residual terms in a fixed proof representation. \textbf{(ii)} We develop exact and
heuristic sparsification procedures, from exhaustive search in small instances
to weighted $\ell_1$-type surrogates for larger ones. \textbf{(iii)} We propose
an SDP search for deriving valid inequalities from those already available in a
PEP formulation, thereby generating candidate intermediate lemmas.
\textbf{(iv)} On gradient-descent and fast-gradient examples, the workflow
recovers weaker fitted interpolation inequalities, a three-hypothesis GD proof,
and compact FGM multiplier patterns. \textbf{(v)} On proximal methods, the same
candidate-lemma workflow recovers compact Lyapunov proofs for the proximal point
residual bound and accelerated proximal point saddle-gap estimate.

\paragraph{Related work and scope.}
The PEP framework was introduced by \cite{drori2014performance} and formalized
by \cite{taylorconvex,taylor2017exact}. Its Lagrangian dual yields certificates
for worst-case performance bounds, and has led to numerous tight
bounds \citep[e.g.,][]{abbaszadehpeivasti2024rate, barre2020principled,
bergthomsen2025tight, de2017worst, dragomir2022optimal, gorbunov2022last,
goujaud2022optimal, rotaru2024exact, taylor2019stochastic}. Related worst-case
and PEP-inspired analyses now also cover operator-splitting and fixed-point
iterations~\citep{park2022exact,ryu2020operator,yoon2024optimal} as well as
min--max algorithms~\citep{shugart2025negative}.

More broadly, PEPs provide a rigorous framework for characterizing the proof
structures inherent to first-order optimization in classical setups
\citep{goujaud2023fundamental}. This has enabled the discovery of proofs that
are difficult to obtain by traditional analysis \citep{kim2016optimized,
kim2018another, kim2018generalizing, lieder2021convergence, ryu2020operator,
upadhyaya2024automated}, as well as the systematic design of new algorithms
\citep{altschuler2025acceleration,drori2020efficient, jang2025computer,
kim2016optimized,taylor2023optimal,upadhyaya2026optimal}. These advancements
are supported by dedicated software tools~\citep{goujaud2024pepit,
taylor2017performance,upadhyaya2025autolyap}.

However, such certificates are rarely unique. The resulting proofs are often
highly complex, and understanding, replicating, or adapting them to new
algorithms often requires substantial effort. This paper demonstrates
that tools from sparse optimization and statistical learning can be used to
make the passage from certificates to proofs more systematic: they provide a way to search for
\emph{simpler} certificates that still recover the desired performance
guarantees.

\paragraph{Organization.}
The rest of the paper is organized as follows.  First, in \Cref{sec:simple-proof-structures}, we formalize proof
structures and define the proof-complexity measures used to compare them. We
then describe search procedures for sparsifying certificates and generating
candidate lemmas in \Cref{sec:heuristic-search}. The experimental examples, given in \Cref{sec:experimental-examples}, show how candidate lemmas recover
fitted interpolation inequalities for gradient descent, while sparsification
exposes compact multiplier patterns for fast-gradient methods. We then show in \Cref{sec:proximal-certificates} that the 
same candidate-lemma search technique
recovers Lyapunov functions for proximal algorithms, yielding sharp
proximal point residual/value bounds and an accelerated proximal point
saddle-gap estimate. The appendices collect the supporting details:
\Cref{app:general-pep-framework} introduces the PEP setting in which these proof
structures arise; \Cref{app:candidate-lemma-sdp-formulation} derives the
candidate-lemma SDP; \Cref{app:example-peps} describes the numerical PEP
searches behind the experimental examples; and
\Cref{app:proximal-certificate-proofs} gives the closed-form proximal
certificate proofs.

\paragraph{Notation.} We write $a\triangleq b$ when $a$ is defined as $b$.
For integers $n,m\in\mathbb{Z}$, write
$\llbracket n,m \rrbracket=\Bigset{i\in\mathbb Z:n\leq i\leq m}$.
For any finite set $S$, $\abs{S}$ denotes its cardinality. We use
$\inner{\cdot}{\cdot}$ for the ambient inner product and $\Bignorm{\cdot}$ for
its induced norm; for symmetric positive-definite $B$, write
$\Bignorm{x}_B^2\triangleq\inner{Bx}{x}$. The trace is denoted by $\tr$.
Let $\mathbb S^r$ be the space of $r\times r$ real symmetric matrices and
$\mathbb S_+^r$ its positive semidefinite cone; $A\succeq0$ means
$A\in\mathbb S_+^r$. We write $\R_+^p$ for the nonnegative orthant. When
$x_\star\in\argmin_{x} f\Bigp{x}$, write $f_\star=f\Bigp{x_\star}$. For convex
functions, $\partial f$ denotes the subdifferential; for concave scalar
penalties, $\partial\phi$ denotes the superdifferential.

%% file: sections/simple_proof_structures.tex
\section{Simple proof structures}\label{sec:simple-proof-structures}

This section fixes the proof representation used throughout the paper. We view a
convergence proof as a decomposition of the target guarantee into nonnegative
valid inequalities and residual terms; this representation makes proof
simplicity measurable through its active inequalities and residual structure.

\subsection{Proof structures}\label{sec:proof-structures}
As described in \eqref{eq:intro-pep}, a convergence proof typically certifies, for each $\hor\in \mathbb{N}$,
a bound of the form
\begin{equation}\label{eq:target}
    \Pc_f\Bigp{\Bigp{x_k}_{k\in\llbracket0,\hor\rrbracket},x_\star}
    \leq \rho(\hor)\,\, \Ic_f\Bigp{\Bigp{x_k}_{k\in\llbracket0,\hor\rrbracket},x_\star},
\end{equation}
uniformly over all problem dimensions, functions $f\in\F$, minimizers
$x_\star$, initial points $x_0$, and trajectories $(x_k)_{k\in\llbracket0,\hor\rrbracket}$ generated by the method $\mathcal{A}$, where
$\Pc_f$ is the chosen performance measure, $\Ic_f$ is an initialization budget, and
$\rho(\hor)$ is the certified rate at iteration $\hor$. The \textit{performance estimation} literature (see e.g., \cite{goujaud2023fundamental}) has studied the structure of first-order convergence
proofs. We let $\labels$ denote the finite set of point labels made available
to the proof; for example, one may take
$\labels=\Bigset{\star,0,\ldots,\hor}$. For $i,j\in\labels$, let $\KK_{i,j}$
be the finite set indexing the scalar hypotheses attached to the ordered pair
$(i,j)$. A key takeaway is that many first-order convergence proofs can be
written as
\begin{equation}
    \rho(\hor) \Ic_f\Bigp{\Bigp{x_k}_{k\in\llbracket0,\hor\rrbracket},x_\star}
    - \Pc_f\Bigp{\Bigp{x_k}_{k\in\llbracket0,\hor\rrbracket},x_\star}
    =\sum_{i,j\in\labels}\sum_{\ell\in\KK_{i,j}}
    \lambda^{\Bigp{\ell}}_{i,j} H^{\Bigp{\ell}}\Bigp{x_i,x_j}
    + \sum_{i=1}^{r} c_i R_i,
    \label{eq:generic-proof} \tag{Proof}
\end{equation}
where $r\in\mathbb N$, the Lagrangian multipliers
$\lambda^{\Bigp{\ell}}_{i,j}$, and the coefficients $c_i$ are all
non-negative, with:
\begin{enumerate}[noitemsep,topsep=0pt]
    \item For each $i,j\in\labels$ and $\ell\in\KK_{i,j}$, the term
    $H^{\Bigp{\ell}}\Bigp{x_i,x_j}$ is an expression provided by the
    definition of the function class and algorithm, and is enforced to be
    non-negative. For the first-order function classes considered here, the
    hypothesis list typically consists of the pairwise interpolation
    inequalities~\cite{taylor2017smooth}, possibly supplemented by additional
    redundant valid inequalities.
    \item Each $R_i$ is a nonnegative residual term, typically a squared norm,
    fixed as part of the proof representation.
\end{enumerate}
Since every summand on the right-hand side is nonnegative, the proof identity \eqref{eq:generic-proof} proves the target bound \eqref{eq:target}.
Classes of problems known to admit such decompositions are detailed in \Cref{app:general-pep-framework}.

\subsection{Proof complexity}\label{sec:proof-complexity}
The proof identity \eqref{eq:generic-proof} gives a direct way to discuss the complexity of
a convergence proof. Once the iteration $\hor$, the point set $\labels$, and the
allowed inequality and residual terms are fixed, the complexity measures below
track which valid inequalities appear with positive multipliers and which
residual terms have positive coefficients.

For a certificate $C$ of the form~\eqref{eq:generic-proof}, define the active
hypothesis and residual sets by
\begin{equation*}
  \begin{aligned}
    \mathcal A^H(C)
    &\triangleq
    \Bigset{(i,j,\ell): i,j\in\labels,\
    \ell\in\KK_{i,j},\
    \lambda^{\Bigp{\ell}}_{i,j}>0,\
    H^{\Bigp{\ell}}\Bigp{x_i,x_j}\not\equiv0},
    \\
    \mathcal A^R(C)
    &\triangleq
    \Bigset{i\in\llbracket1,r\rrbracket:c_i>0}.
  \end{aligned}
\end{equation*}
We call the pair $\Bigp{\mathcal A^H(C),\mathcal A^R(C)}$ the
\textit{active pattern} of $C$. Thus
$(i,j,\ell)\in\mathcal A^H(C)$ means that the proof uses the hypothesis
$H^{\Bigp{\ell}}\Bigp{x_i,x_j}\geq0$, while
$i\in\mathcal A^R(C)$ means that the residual term $R_i$ appears in the
residual decomposition. These active sets induce the two complexity measures used below:
\begin{enumerate}
    \item the \textbf{hypothesis complexity} of a proof identity is the number of active inequalities $\abs{\mathcal A^H(C)}$,
    \item the \textbf{residual complexity} of a proof identity is the number of active residual terms $
        \abs{\mathcal A^R(C)}$.
\end{enumerate}

For numerical experiments, the active sets are thresholded:
$\lambda^{\Bigp{\ell}}_{i,j}>0$ and $c_i>0$ are replaced by
$\lambda^{\Bigp{\ell}}_{i,j}>\varepsilon_{\mathrm{act}}$ and
$c_i>\varepsilon_{\mathrm{act}}$, with the tolerance specified in the
corresponding experiment or table.

%% file: sections/heuristic_search.tex
\section{Heuristic search for simple proofs}\label{sec:heuristic-search}

We now introduce concrete search procedures that use the proof-complexity
measures above to find simpler certificates. Throughout this section, we fix a
convergence guarantee, a finite list of available valid hypotheses, and a
residual decomposition. The goal is to find another certificate for essentially the same
guarantee that uses fewer active hypotheses, simpler residual terms, or both.
We write
\[
    \mathcal J
    \triangleq
    \Bigset{(i,j,\ell): i,j\in\labels,\ 
    \ell\in\KK_{i,j}}
\]
for the available hypothesis indices. For
$h=(i,j,\ell)\in\mathcal J$, abbreviate
$H_h\triangleq H^{\Bigp{\ell}}\Bigp{x_i,x_j}$ and
$\lambda_h\triangleq\lambda^{\Bigp{\ell}}_{i,j}$.

There are two useful choices to make before applying these heuristics:
\begin{enumerate}
    \item \textbf{Expand the hypothesis set.} When available, add redundant but
    valid inequalities (e.g., the descent lemma)
    before sparsifying. This larger finite list fixes what counts as an
    \textit{available hypothesis} and can eventually lead to simpler certificates, even though the
    search starts with more possible terms.
    \item \textbf{Relax the target.} Let $\rho_\star>0$ denote the best rate
    certified by the chosen formulation. For a target relative suboptimality
    tolerance $\varepsilon_{\mathrm{rel}}\geq0$, set $\bar\rho \triangleq \Bigp{1+\varepsilon_{\mathrm{rel}}}\rho_\star$.
\end{enumerate}

\subsection{Exhaustive sparsification}\label{sec:exhaustive-sparsification}
Finding a proof with minimal hypothesis complexity is a best-subset-type
problem over the available hypotheses, analogous to classical best subset
selection in sparse regression~\citep{bertsimas2016best}. Closely related
minimum-cardinality feasibility and sparse-approximation problems are
NP-hard~\citep{amaldi1998approximability,natarajan1995sparse}. Exhaustive
sparsification has the same combinatorial character: for each candidate active
pattern, one must check whether there exists a certificate using no hypotheses
outside that pattern. Exhaustive sparsification is the
direct combinatorial baseline. It is useful as a ground-truth benchmark on small
instances, but its cost still scales exponentially with the number of possible active hypothesis patterns.

In the indexed setting of~\eqref{eq:generic-proof}, a candidate active
pattern is a pair $\Pi=\Bigp{I_H,I_R}$ with
$I_H\subseteq\mathcal J$ and $I_R\subseteq\llbracket1,r\rrbracket$. Define
\[
    \rho\Bigp{\Pi}
    \triangleq
    \inf\Bigset{\rho\in\R:
        \exists\text{ a certificate }C\text{ of rate }\rho
        \text{ such that }
        \mathcal A^H(C)\subseteq I_H,\ 
        \mathcal A^R(C)\subseteq I_R}.
\]
By convention, $\rho\Bigp{\Pi}=+\infty$ if the set above is empty, i.e., if no
finite-rate certificate with that active-pattern restriction exists. Given a
tolerance level $\bar\rho$, call $\Pi$ admissible if
$\rho\Bigp{\Pi}\leq\bar\rho$. Since restricting the active
pattern cannot improve the optimum, $\rho\Bigp{\Pi}\geq\rho_\star$. With this
notation, exhaustive sparsification keeps all residual terms available and computes
\[
    h_\star\triangleq
    \min\Bigset{\abs{I_H}:
        I_H\subseteq\mathcal J,\ 
        \rho\Bigp{\Bigp{I_H,\llbracket1,r\rrbracket}}\leq\bar\rho},
\]
then returns all admissible hypothesis patterns $I_H$ of size $h_\star$.

\subsection{Sparse minimization heuristics}\label{sec:sparse-minimization-heuristics}
In this section, for clarity of exposition, we focus on \textit{hypothesis complexity}. The same ideas can be applied to \textit{residual complexity}, for example by penalizing the rank of residual slack matrices through log-det heuristics~\citep{fazel2003log}, but we do not detail them here.

The exhaustive baseline is exact but quickly becomes too expensive. A scalable
alternative is to optimize a \textit{sparsity surrogate} over the set of
certificates that prove the relaxed target. Let $\mathcal C_{\bar\rho}$ denote the certificates in
the fixed search space that certify a bound no larger than $\bar\rho$. Each such
certificate $C\in\mathcal C_{\bar\rho}$ is defined by nonnegative hypothesis
multipliers $(\lambda_h)_{h\in\mathcal J}$ and active hypothesis set
$\mathcal A^H(C)=\Bigset{h\in\mathcal J:
\lambda_h>0,\ H_h\not\equiv0}$.

We consider separable penalties $\phi_h:\R_+\to\R$ applied to these
multipliers:
\begin{equation}
    \label{eq:separable-sparse-penalty}
    \begin{aligned}
    \problemhead{minimize}{C\in \mathcal C_{\bar\rho}}
        &
        \sum_{h\in\mathcal J}\phi_h\Bigp{\lambda_h}.
    \end{aligned}
\end{equation}
The choices below approximate the number of active hypothesis multipliers. A natural first choice is
\begin{equation*}
    \tag{Plain $\ell_1$}
    \phi_h\Bigp{x}=x,
\end{equation*}
which is the standard lasso-type convex relaxation of
sparsity~\citep{tibshirani1996regression}.
Second, a common nonconvex surrogate for the number of nonzero coordinates is the
log-sum penalty:
\begin{equation*}
    \tag{Log-sum}
    \phi_h\Bigp{x}=\log\Bigp{x+\delta},
\end{equation*}
where $\delta > 0$ stabilizes the logarithm near zero.
Its first-order majorization gives reweighted $\ell_1$
iterations~\citep{candes2008enhancing}.
Because the attainable ranges of the multipliers can differ by orders of
magnitude, an unnormalized log-sum surrogate can impose much larger effective
shrinkage on some multipliers than on others. To remedy this, a common strategy is to introduce
multiplier-specific normalizations.
\begin{equation*}
    \tag{Normalized log-sum}
    \phi_h\Bigp{x}=\log\Bigp{\frac{x}{M_h}+\delta},
\end{equation*}
where $M_h > 0$ is a reference scale for the multiplier of
$H_h\geq0$, estimated from its feasible range and clipped to a
positive bounded interval in the experiments below.
To distinguish active multipliers rather than their magnitudes, one can use a
capped penalty:
\begin{equation*}
    \tag{Capped $\ell_1$}
    \phi_h\Bigp{x}= \min\Bigset{\theta,\frac{x}{M_h}}.
\end{equation*}
The cap is linear for $x/M_h<\theta$ and then saturates, so $\theta$ is the
normalized activation threshold. Thus the objective behaves like a scaled
active-multiplier count rather than a magnitude penalty. This is the standard capped-$\ell_1$ sparsity
surrogate handled by multi-stage convex
relaxation~\citep{zhang2010analysis}.

In practice, nonlinear choices in~\eqref{eq:separable-sparse-penalty} are optimized by
iterative majorization: at each step, we replace the penalty by a linear upper
bound at the current multipliers. Given weights $w$, define
\begin{equation}
    \label{eq:l1-reweighted}
    \begin{aligned}
        \textsc{Weighted-Cert}\Bigp{w,\bar\rho}
        \in
        \argmin_{C\in\mathcal C_{\bar\rho}}
        \sum_{h\in\mathcal J} w_h\lambda_h.
    \end{aligned}
\end{equation}
returning a minimizing certificate and its hypothesis multipliers. 
The optimization problem in \eqref{eq:l1-reweighted} is convex and thus solvable efficiently. 
Along iterations, the relaxed target $\bar \rho$ stays fixed; only the selection criterion changes.
For the concave penalties above, the linear majorization weights are chosen from
the corresponding superdifferentials:
$w_h^{[m+1]}\in\partial\phi_h(\lambda_h^{[m]})$.
For normalized log-sum,
$w_h^{[m+1]}=(\lambda_h^{[m]}+\delta M_h)^{-1}$, while capped $\ell_1$ gives
$w_h^{[m+1]}=M_h^{-1}$ below the cap and $w_h^{[m+1]}=0$ above it,
with any supergradient between these values at the threshold.
The resulting procedure is summarized in \Cref{alg:l1-simplification}. In the
algorithm, bracketed superscripts denote iteration counters, so they do not
conflict with the hypothesis superscript $\ell$.

\algtext*{EndFor}
\begin{algorithm}
    \caption{Majorized sparsification heuristic}
    \label{alg:l1-simplification}
    \begin{algorithmic}[1]
        \Require Relaxed target $\bar\rho$, coordinate penalties $\phi_h$, initial weights $w^{[0]}$, number of iterations $T$
        \For{$m\in\llbracket 0,T-1 \rrbracket$}
            \State $C^{[m]}\gets \textsc{Weighted-Cert}\Bigp{w^{[m]},\bar\rho}$
            \State Let $\lambda^{[m]}$ denote the corresponding multipliers
            \State Choose $w_h^{[m+1]}\in\partial\phi_h\Bigp{\lambda_h^{[m]}}$, $\forall h\in\mathcal J$
        \EndFor
        \State \Return $C^{[T-1]}$.
    \end{algorithmic}
\end{algorithm}
The output certificate proves the relaxed target, and
$\mathcal A^H\Bigp{C^{[T-1]}}$ is the selected sparsification pattern.

\subsection{Discovering proof structures through intermediate lemmas}\label{sec:automatic}

The previous heuristics simplify a certificate after the list of available
hypotheses has been fixed. We can also enlarge this list by searching for
additional valid inequalities that may serve as reusable steps in a shorter
proof.

Let $\mathcal L_{\mathrm{cand}}$ denote the candidate-lemma index set. For each
$\kappa\in\mathcal L_{\mathrm{cand}}$, let
$\mathcal I_\kappa\subseteq\mathcal J$ be the hypotheses allowed in its short
proof, and let $\Psi_\kappa$ be the proposed auxiliary inequality. In the
notation of \eqref{eq:generic-proof}, this means that
\[
    \Psi_\kappa\geq0,\qquad
    \Psi_\kappa
    =
    \sum_{h\in\mathcal I_\kappa}
        \alpha_h^{\Bigp{\kappa}} H_h
    +
    \sum_{i=1}^{r_\kappa} \beta_i^{\Bigp{\kappa}} R_i^{\Bigp{\kappa}},
    \qquad
    \alpha_h^{\Bigp{\kappa}},\beta_i^{\Bigp{\kappa}}\geq0.
\]
Thus the validity of $\Psi_\kappa\geq0$ follows from the same proof structure as the
original certificate. A candidate lemma becomes an
\emph{intermediate lemma} only when it is used with positive multiplier in the
simplified proof identity, for example
\[
    \rho(\hor)\Ic_f-\Pc_f
    =
    \sum_{h\in\mathcal J}\lambda_hH_h
    +
    \sum_{\kappa\in\mathcal L_{\mathrm{cand}}}\eta_\kappa\Psi_\kappa
    +
    \sum_{i=1}^r c_i R_i,
    \qquad
    \lambda_h,\eta_\kappa,c_i\geq0.
\]
Substituting the selected candidate lemmas back into this identity recovers a
certificate of the original form; the intermediate lemmas simply expose useful
derived inequalities as separately checkable steps. The SDP formulation and
extraction procedure for candidate lemmas are deferred to
\Cref{app:candidate-lemma-sdp-formulation}.

%% file: sections/experimental_examples.tex
\section{Experimental examples}\label{sec:experimental-examples}
The examples below apply the proposed simplification procedures to several
different problems. In each case, we start from numerical certificate searches
and extract explicit proof identities or multiplier patterns from their output.
We begin with one-step gradient descent, where the classical function-value
contraction admits a short proof despite dense raw certificates. We then consider
fast-gradient methods, where sparsification reduces the hypothesis complexity.
Finally, the proximal examples show how the candidate-lemma SDP recovers
Lyapunov functions and one-step inequalities that yield tight proofs for the
proximal point method and its accelerated variant in the monotone-operator
setting.

The notebooks and code implementing the procedures used in these examples are
available in the public source-code repository
\url{https://github.com/DanielBergThomsen/simple-proofs}.
All reported experiments were run on a MacBook Pro
with an Apple M4 Max chip, 14 CPU cores, and 36 GB of memory. The examples are
small deterministic SDP and enumeration computations; except for the largest
exhaustive active-set checks, they are not computationally intensive.

\subsection{Gradient descent}\label{sec:gradient-descent}
Consider one step of gradient descent,
\begin{equation}
    x_1 = x_0 - \gamma \nabla f\Bigp{x_0},
    \tag{GD}
\end{equation}
where $\gamma > 0$ is the stepsize. We are interested in worst-case guarantees on
the functional residual after one step on the class $\F_{\mu,L}$, $0\leq\mu<L$,
under the initial normalization $f\Bigp{x_0}-f_\star \leq 1$. The corresponding
PEP and SDP formulations are collected in
\Cref{app:pep-gd}.

For such functions $f \in \F_{\mu,L}$, and for any pair of points $u,v\in\R^d$, define
\begin{equation}
\begin{aligned}
    H_{\mu,L}\Bigp{u,v}
    &\triangleq
    f\Bigp{u}-f\Bigp{v}
    -\inner{\nabla f\Bigp{v}}{u-v}
    -\frac{\mu}{2}\Bignorm{u-v}^2
    \\
    &\quad
    -\frac{1}{2\Bigp{L-\mu}}
    \Bignorm{
        \nabla f\Bigp{u}-\nabla f\Bigp{v}-\mu\Bigp{u-v}
    }^2
\end{aligned}
    \label{eq:gd-interpolation-inequality}
\end{equation}
By the two-point interpolation property for $\F_{\mu,L}$~\citep{taylor2017smooth},
$H_{\mu,L}\Bigp{u,v}\geq0$ for all points $u,v$.

\subsubsection{Exhaustive sparsification}\label{sec:gd-exhaustive-sparsification}
For the one-step functional-residual experiment, exhaustive enumeration is still
practical because the points $x_\star$, $x_0$, and $x_1$ give rise to only six
nontrivial two-point interpolation inequalities.
\Cref{fig:gd-functional-residual-exhaustive} compares the raw certificate
weights with the weights retained after exhaustive sparsification. The retained
weights multiply $H_{\mu,L}\Bigp{x_\star,x_0}$,
$H_{\mu,L}\Bigp{x_\star,x_1}$, and $H_{\mu,L}\Bigp{x_0,x_1}$; the weights on
the other three two-point inequalities vanish. They follow the branchwise
closed forms below for
$0<\gamma\leq 2/L$, with threshold $\gamma = \frac{2}{L+\mu}$:
\begin{equation*}
    \begin{aligned}
        \lambda_{\star,0}
        &=
        \begin{cases}
            \gamma\mu\Bigp{1-\gamma\mu}, & \gamma \leq \frac{2}{L+\mu}, \\
            \Bigp{2-\gamma L}\Bigp{\gamma L-1}, & \gamma \geq \frac{2}{L+\mu},
        \end{cases}
        \qquad
        \begin{aligned}
            \lambda_{\star,1} &= \min\!\Bigset{\gamma\mu,\ 2-\gamma L},\\
            \lambda_{0,1} &= \max\!\Bigset{1-\gamma\mu,\ \gamma L-1}.
        \end{aligned}
    \end{aligned}
\end{equation*}

\begin{figure}[t]
    \centering
    \includegraphics[width=.79\linewidth]{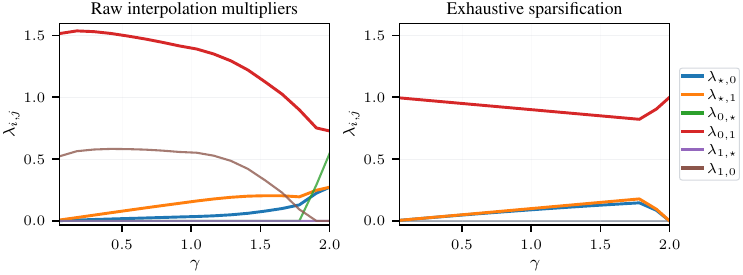}
    \caption{One-step gradient descent with functional-residual normalization. Left: raw certificate weights on the stepsize grid. Right: the retained weights after exhaustive sparsification.}
    \label{fig:gd-functional-residual-exhaustive}
\end{figure}

\subsubsection{Fitted curvatures for the classical rate of GD}\label{sec:gd-automatic-hypotheses}
Applying the candidate-lemma search of \Cref{sec:automatic} with one singleton
candidate for each two-point inequality recovers a tight one-step proof of the
GD function-value guarantee for every stepsize $0<\gamma\leq 2/L$. The rate is
the classical fixed-step function-value contraction; here, the search
rediscovers it as a three-hypothesis certificate.
Concretely, the singleton candidate lemmas can be recognized as
interpolation inequalities for the fitted larger class
\begin{equation}
    \F_{\mu,L}
    \subseteq
    \F_{\mu_\gamma,L_\gamma},
    \qquad
    \Bigp{\mu_\gamma,L_\gamma}
    \triangleq
    \begin{cases}
        \Bigp{\mu,\ \frac{2}{\gamma}-\mu},
        & 0<\gamma\leq\frac{2}{L+\mu},\\[1mm]
        \Bigp{\frac{2}{\gamma}-L,\ L},
        & \frac{2}{L+\mu}\leq\gamma\leq\frac{2}{L}.
    \end{cases}
    \label{eq:gd-fitted-curvatures}
\end{equation}
\Cref{fig:gd-fitted-curvature-scan} shows the fitted curvatures extracted by a
candidate-lemma SDP after fixing the sparse certificate weights on a
representative grid; the measured points lie on the closed-form branches in
\eqref{eq:gd-fitted-curvatures}.
For $0<\gamma\leq2/L$, these constants satisfy
$0\leq\mu_\gamma\leq\mu<L\leq L_\gamma$, and
$\gamma=2/\Bigp{L_\gamma+\mu_\gamma}$; hence the fitted inequalities are valid
for every original $\F_{\mu,L}$ instance, but impose only the weaker
requirements of the larger class.

\begin{figure}[t]
    \centering
    \includegraphics[width=.79\linewidth]{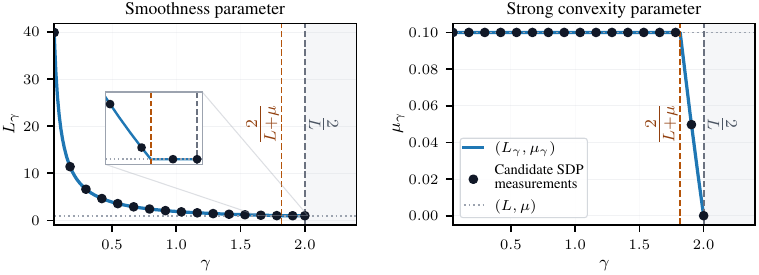}
    \caption{Fitted interpolation curvatures identified from singleton
    candidate lemmas for the one-step gradient descent certificate, shown for
    $L=1$ and $\mu=0.1$.
    The SDP measurements follow the closed-form curves: before the threshold
    $2/(L+\mu)$, the fitted class keeps $\mu_\gamma=\mu$ and increases
    $L_\gamma$; after the threshold, it keeps $L_\gamma=L$ and weakens
    $\mu_\gamma$. Dotted horizontal lines mark the original constants.}
    \label{fig:gd-fitted-curvature-scan}
\end{figure}

Define
\begin{equation}
    \begin{gathered}
    \rho_\gamma
    \triangleq
    \max\Bigset{\Bigp{1-\gamma\mu}^2,\Bigp{\gamma L-1}^2}
    =
    \Bigp{1-\gamma\mu_\gamma}^2
    =
    \Bigp{\gamma L_\gamma-1}^2,
    \\
    \Bigp{\lambda_{\star,0},\lambda_{\star,1},\lambda_{0,1}}
    \triangleq
    \Bigp{\gamma\mu_\gamma\sqrt{\rho_\gamma},
    \gamma\mu_\gamma,\sqrt{\rho_\gamma}}.
    \end{gathered}
    \label{eq:gd-certificate-ingredients}
\end{equation}
The three $\lambda$'s are nonnegative for $0<\gamma\leq2/L$. On the large-step
branch, the fitted value
$\mu_\gamma=2/\gamma-L$ is exactly the weaker strong-convexity parameter used by
\citet[after Eq.~(1.3)]{uschmajew2022note}; the proof below uses the same
weakening trick but covers the full range $0<\gamma\leq2/L$.

\begin{theorem}[Worst-case performance of one-step GD]
\label{thm:gd-fitted-certificate}
Let $f\in\F_{\mu,L}$ with $0\leq\mu<L$, and let
$0<\gamma\leq2/L$. With the quantities in
\eqref{eq:gd-fitted-curvatures}--\eqref{eq:gd-certificate-ingredients},
one step of gradient descent satisfies
$f\Bigp{x_1}-f_\star\leq\rho_\gamma\Bigp{f\Bigp{x_0}-f_\star}$.
\end{theorem}

\begin{proof}
The inclusion $\F_{\mu,L}\subseteq\F_{\mu_\gamma,L_\gamma}$ makes the three
$H$-terms below nonnegative.
Substituting
$x_1-x_\star=x_0-x_\star-\gamma\nabla f\Bigp{x_0}$ and
$\gamma=2/\Bigp{L_\gamma+\mu_\gamma}$ gives the certificate identity
\[
\begin{aligned}
    \rho_\gamma\Bigp{f\Bigp{x_0}-f_\star}
    -\Bigp{f\Bigp{x_1}-f_\star}
    &=
    \lambda_{\star,0}H_{\mu_\gamma,L_\gamma}\Bigp{x_\star,x_0}
    +\lambda_{\star,1}H_{\mu_\gamma,L_\gamma}\Bigp{x_\star,x_1}
    \\
    &\quad
    +\lambda_{0,1}H_{\mu_\gamma,L_\gamma}\Bigp{x_0,x_1}
    +R_\gamma.
\end{aligned}
\]
Here
\[
    R_\gamma
    \triangleq
    \frac{\gamma}{4\sqrt{\rho_\gamma}}
    \Bignorm{
        \nabla f\Bigp{x_0}+\nabla f\Bigp{x_1}
        -\gamma\mu_\gamma L_\gamma\Bigp{x_0-x_\star}
    }^2 .
\]
Since the multipliers and $R_\gamma$ are also nonnegative, this
proves the claim.
\end{proof}

\begin{remark}[Comparison with original curvatures]
\label{rem:gd-original-curvatures}
For $0<\mu<L$, setting the nonsmooth term $h$ to zero in the
proof of \citet[Theorem~3.3]{taylor2018proximal} recovers
this statement. Using the same three active multipliers with the
original $\F_{\mu,L}$ interpolation inequalities gives instead
\[
\begin{aligned}
    \rho_\gamma\Bigp{f\Bigp{x_0}-f_\star}
    -\Bigp{f\Bigp{x_1}-f_\star}
    &=
    \lambda_{\star,0}H_{\mu,L}\Bigp{x_\star,x_0}
    +\lambda_{\star,1}H_{\mu,L}\Bigp{x_\star,x_1}
    \\
    &\quad
    +\lambda_{0,1}H_{\mu,L}\Bigp{x_0,x_1}
    +R_\gamma'.
\end{aligned}
\]
Specializing the branchwise slack decomposition from that proof to $h\equiv0$ gives a
nonnegative but more complicated residual; its closed forms on the two
stepsize branches are given in \Cref{app:gd-original-residuals}.
\end{remark}

\subsection{Fast gradient method}\label{sec:fast-gradient-method}
For a horizon $N$, we consider the following parameterization of the
smooth-convex fast gradient method, with the $k$-indexed updates taken over
$k\in\llbracket 0,N-2 \rrbracket$:
\begin{equation}
    \begin{alignedat}{2}
        y_0 &= x_0,\qquad&
        x_N &= y_{N-1} - \frac{1}{L}\nabla f\Bigp{y_{N-1}}, \\
        x_{k+1} &= y_k - \frac{1}{L}\nabla f\Bigp{y_k},\qquad&
        y_{k+1} &= x_{k+1} + \frac{k}{k+3}\Bigp{x_{k+1} - x_k}.
    \end{alignedat}
    \tag{FGM}
\end{equation}
In contrast with the GD example,
this one is posed on the smooth-convex class $\F_L$, with distance
normalization $\Bignorm{x_0-x_\star}^2 \leq 1$ and terminal performance measure
$f\Bigp{x_N}-f_\star$. The corresponding PEP and SDP formulations are collected in
\Cref{app:pep-fgm-cvx}.

Let $H_L\Bigp{u,v}\triangleq H_{0,L}\Bigp{u,v}$;
for any $f\in\F_L$, one has $H_L\Bigp{u,v}\geq0$ for all
points $u,v$.

\paragraph{Target rate.} Under this normalization, the conjectured optimal
worst-case rate for the smooth-convex FGM endpoint is
$\rho_\star\Bigp{N,L}=\frac{2L}{N^2+5N+6}$~\citep[Table~1]{taylor2017exact}.
The same table gives the relaxed FGM target
$\rho_{\mathrm{rel}}\Bigp{N,L}=\frac{2L}{N^2+5N+2}$.
The relative suboptimality criterion is set so that this relaxed rate can be
recovered, allowing an additional tolerance for numerical noise.

\subsubsection{\texorpdfstring{$\ell_1$}{l1}-heuristics}\label{sec:fgm-l1-heuristics}
For the computer-aided sparsity comparison we use only the base
smooth-convex inequalities $H_L$ at $N=3$. The points
$\Bigset{x_\star,y_0,y_1,y_2,x_3}$ generate the ordered two-point inequalities
used in the search.
At this size, exhaustive enumeration is still practical and provides a useful
calibration for the penalty-based searches. Plain $\ell_1$ and log-sum are
scale-sensitive on this instance and retain denser patterns, while the normalized
log-sum heuristic from
\Cref{sec:sparse-minimization-heuristics} moves the returned pattern closer to
the exhaustive one. The active-inequality patterns for the $N=3$ comparison are
detailed in \Cref{app:fgm-multiplier-support}. The same comparison across longer
horizons is shown in \Cref{fig:fgm-cardinality-comparison}.

\begin{figure}[t]
    \centering
    \includegraphics[width=.79\linewidth]{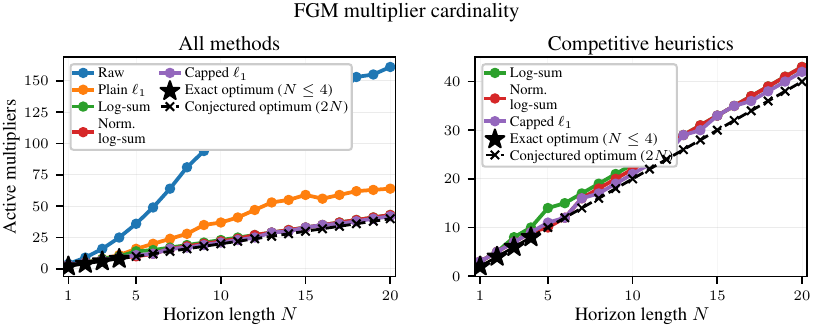}
    \caption{
    FGM hypothesis complexity across horizon lengths: all methods on
    the left, and the competitive continuous sparsification heuristics on the
    right, which track the exact or conjectured sparse active-inequality pattern
    more closely.}
    \label{fig:fgm-cardinality-comparison}
\end{figure}

%% file: sections/proximal_certificates.tex
\section{Proximal certificates from candidate lemmas}
\label{sec:proximal-certificates}
The candidate-lemma SDP can return a local Lyapunov inequality rather than
only a shorter terminal certificate \citep{rotaru2026thesis,taylor2019stochastic}.
In the two examples below, the SDP identifies both the potential and its
one-step decrement for the proximal point method and its accelerated variant;
the closed-form proofs are given in \Cref{app:proximal-certificate-proofs}.

\subsection{Proximal point residuals}
Let $F:\R^d\to\R\cup\{+\infty\}$ be closed, proper, and convex, and let
$x_\star\in\argmin F$ with $F_\star=F\Bigp{x_\star}$. Fix a positive-definite
$B\in\mathbb S^d$.
For positive stepsizes $\alpha_1,\ldots,\alpha_N$, the proximal point method is
\begin{equation}
    x_k\in
    \argmin_{x\in\R^d}
    \Bigset{
        F\Bigp{x}+\frac{1}{2\alpha_k}\Bignorm{x-x_{k-1}}_B^2
    },
    \qquad
    g_k\triangleq\frac{B\Bigp{x_{k-1}-x_k}}{\alpha_k}\in\partial F\Bigp{x_k}.
    \tag{PPM}
\end{equation}
Writing $A_N=\sum_{k=1}^N\alpha_k$, the recovered certificate gives a
considerably simpler Lyapunov-function proof of the sharp last-residual
estimate conjectured by \citet{taylor2017exact} and proved by
\citet{guyang2023tight}; the same Lyapunov function yields the value bound.
\begin{restatable}[PPM residual and value bounds]{theorem}{pparesidualvaluethm}
\label{thm:ppa-residual-value}
If $\Bignorm{x_0-x_\star}_B\leq\Delta_0$, then for every $N\geq1$,
\[
    \Bignorm{g_N}_{B^{-1}}\leq\frac{\Delta_0}{A_N},
    \qquad
    F\Bigp{x_N}-F_\star\leq\frac{\Delta_0^2}{4A_N}.
\]
\end{restatable}

\subsection{Accelerated proximal point saddle gaps}
The second example recovers an operator potential for the accelerated proximal
point method and then
applies it to a saddle subdifferential. Let $\HH_1,\HH_2$ be real Hilbert
spaces, and let $\Phi$ be a saddle function with
$\Phi\Bigp{\cdot,v}$ and $-\Phi\Bigp{u,\cdot}$ closed, proper, and convex. Assume that
the saddle subdifferential
$\mathcal M\Bigp{u,v}=\Bigp{\partial\Phi\Bigp{\cdot,v}\Bigp{u},
\partial\Bigp{-\Phi\Bigp{u,\cdot}}\Bigp{v}}$ is maximally monotone on
$\HH_1\times\HH_2$ and that
$0\in\mathcal M\Bigp{u_\star,v_\star}$. For $\beta>0$, initialize
$x_0=y_0=y_{-1}=\Bigp{u_0,v_0}$ and run the accelerated proximal point method
of \citet{kim2021accelerated},
\begin{equation}
    x_{k+1}=\Bigp{\operatorname{Id}+\beta\mathcal M}^{-1}\Bigp{y_k},
    \qquad
    y_{k+1}=x_{k+1}
    +\frac{k}{k+2}\Bigp{x_{k+1}-x_k}
    -\frac{k}{k+2}\Bigp{x_k-y_{k-1}},
    \tag{APPM}
\end{equation}
where the saddle-gap bound below was conjectured.
\begin{restatable}[APPM saddle-gap bound]{theorem}{appmsaddlegapthm}
\label{thm:appm-saddle-gap}
Writing $x_k=\Bigp{u_k,v_k}$, assume the saddle-gap quantities below are finite
along the generated sequence. Then, for every $N\geq1$,
\[
    \Phi\Bigp{u_N,v_\star}-\Phi\Bigp{u_\star,v_N}
    \leq
    \frac{\Bignorm{u_0-u_\star}^2+\Bignorm{v_0-v_\star}^2}{4\beta N}.
\]
\end{restatable}

%% file: sections/conclusion.tex
\section{Conclusion}\label{sec:conclusion}

This paper treats PEP dual solutions not only as tight worst-case certificates,
but also as objects for searching over simpler proof structures; active multipliers
and residual terms provide tangible complexity measures for deciding when a
computer-generated certificate can be made readable. In the examples,
exhaustive sparsification exposes closed-form GD multiplier patterns, candidate
lemmas recover proximal Lyapunov decrements for residual/value and saddle-gap
bounds, and normalized sparsity heuristics reduce FGM active multipliers,
suggesting a workflow in which computer-aided worst-case analysis is followed
by certificate simplification to guide proof design.

%% file: sections/appendix.tex
\section{General PEP framework}\label{app:general-pep-framework}

This section gives the sampled PEP derivation behind the proof structures used
in the main text. In particular, it explains why first-order PEP dual
certificates naturally take the form of identities built from valid inequalities
evaluated on finitely many sampled points.

\subsection{Sampled-data PEP}\label{app:sampled-data-pep}

Fix a horizon $\hor$ and let $\labels$ be the finite set of point labels used in
the proof identity, as in \Cref{sec:proof-structures}. The sampled data are
\[
  \mathcal D
  =
  \Bigset{\Bigp{x_i,f_i,g_i}}_{i\in\labels},
  \qquad
  f_i=f\Bigp{x_i},\quad g_i\in\partial f\Bigp{x_i}.
\]
When a formulation needs additional method-state variables or auxiliary points,
we append them to $\mathcal D$. Such variables are included in $\labels$ only
when oracle data are sampled there and interpolation inequalities are imposed on
them. Equations used only to define added variables, such as algorithm
recurrences, are substituted away in the sampled expressions.

We use the same initialization and performance functionals as in the main text,
and write them as $\Ic_f\Bigp{\mathcal D}$ and
$\Pc_f\Bigp{\mathcal D}$. For the first-order function classes and algorithms
considered in the main text, we use the following finite hypothesis form.

\begin{assumption}[Finite hypothesis form]\label{ass:appendix-class-form}
  For the fixed method, horizon, and label set, the admissibility requirements
  \[
    f\in\F,\qquad
    x_\star\in\argmin_{x\in\R^d} f\Bigp{x},\qquad
    \Bigp{x_k}_{k\in\llbracket0,\hor\rrbracket}
    \in\mathcal T_{\A}\Bigp{f,\hor,x_0}
  \]
  are equivalently represented at the sampled-data level by a finite family of
  scalar hypotheses
  \[
    H^{\Bigp{\ell}}\Bigp{x_i,x_j}\geq0,\qquad
    \forall i,j\in\labels,\quad
    \ell\in\KK_{i,j}.
  \]
\end{assumption}

This is the same hypothesis notation as in \Cref{sec:proof-structures}. Adding
valid redundant hypotheses can change the certificate representation but not the
exact PEP; omitting class-defining interpolation inequalities or required
algorithmic constraints gives a relaxation.

With this notation, the introductory worst-case problem~\eqref{eq:intro-pep}
admits the following finite sampled-data form:
\begin{equation}
  \begin{aligned}
    \problemhead{maximize}{\mathcal D,\ d\geq1}
    &
    \Pc_f\Bigp{\mathcal D} \\
    \problemsubjectto
    & \Ic_f\Bigp{\mathcal D}\leq1,\\
    & H^{\Bigp{\ell}}\Bigp{x_i,x_j}\geq0,\quad
    \forall i,j\in\labels,\quad
    \ell\in\KK_{i,j}.
  \end{aligned}
  \tag{f-PEP}
  \label{eq:f-pep}
\end{equation}

\subsection{Lifted SDP and Lagrangian dual}\label{app:indexed-lifted-sdp-dual}

The sampled PEP above is finite, but it is still written in terms of vectors in
an arbitrary dimension. In the examples considered here, every scalar quantity
depends on those vectors only through their inner products, together with the
sampled function values. We therefore use the standard SDP lift: function values
are collected as free scalar coordinates, and inner products are collected in a
positive semidefinite Gram matrix. The following assumption records the finite
lifted form used below.

\begin{assumption}[Gram representability]
  \label{ass:appendix-finite-lifted-representation}
  For the fixed horizon and label set $\labels$, the sampled quantities can be
  encoded by scalar coordinates $s\in\R^p$ and a Gram matrix $G\succeq0$ such
  that the lifted expressions
  \[
    P\Bigp{s,G}
    \triangleq
    \inner{s}{q_P}+\tr\Bigp{GQ_P},
    \qquad
    I\Bigp{s,G}
    \triangleq
    \inner{s}{q_I}+\tr\Bigp{GQ_I},
  \]
  represent $\Pc_f\Bigp{\mathcal D}$ and
  $\Ic_f\Bigp{\mathcal D}$, respectively, and, for every
  $i,j\in\labels$ and $\ell\in\KK_{i,j}$, the lifted
  expression
  \[
    H^{\Bigp{\ell}}_{i,j}\Bigp{s,G}
    \triangleq
    \inner{s}{q^{\Bigp{\ell}}_{i,j}}+\tr\Bigp{GQ^{\Bigp{\ell}}_{i,j}}
  \]
  represents the sampled hypothesis
  $H^{\Bigp{\ell}}\Bigp{x_i,x_j}$.
\end{assumption}

Under \Cref{ass:appendix-finite-lifted-representation}, the
lifted SDP is
\begin{equation}
  \begin{aligned}
    \problemhead{maximize}{s\in\R^p,\ G\succeq0}
    &
    P\Bigp{s,G} \\
    \problemsubjectto
    & I\Bigp{s,G}\leq1,\\
    & H^{\Bigp{\ell}}_{i,j}\Bigp{s,G}\geq0,
    \qquad \forall i,j\in\labels,\quad
    \ell\in\KK_{i,j}.
  \end{aligned}
  \tag{SDP-PEP}
  \label{eq:indexed-sdp-pep}
\end{equation}

The Lagrangian dual of this reduced SDP is
\begin{equation}
  \begin{aligned}
    \problemhead{minimize}{\substack{\tau\geq0,\\
    \lambda^{\Bigp{\ell}}_{i,j}\geq0,\quad
    i,j\in\labels,\quad
    \ell\in\KK_{i,j}}}
    &
    \tau \\
    \problemsubjectto
    & \tau q_I-q_P
    -\sum_{i,j\in\labels}\sum_{\ell\in\KK_{i,j}}
    \lambda^{\Bigp{\ell}}_{i,j}q^{\Bigp{\ell}}_{i,j}=0,\\
    & \tau Q_I-Q_P
    -\sum_{i,j\in\labels}\sum_{\ell\in\KK_{i,j}}
    \lambda^{\Bigp{\ell}}_{i,j}Q^{\Bigp{\ell}}_{i,j}\succeq0.
  \end{aligned}
  \tag{D-SDP-PEP}
  \label{eq:indexed-d-sdp-pep}
\end{equation}
Feasibility of this dual gives a proof identity of the form used in
\eqref{eq:generic-proof}. Indeed, if
\[
  R
  =
  \tau Q_I-Q_P
  -\sum_{i,j\in\labels}\sum_{\ell\in\KK_{i,j}}
  \lambda^{\Bigp{\ell}}_{i,j}Q^{\Bigp{\ell}}_{i,j}\succeq0,
\]
then the scalar equality gives
\[
  \tau I\Bigp{s,G}
  -P\Bigp{s,G}
  =
  \sum_{i,j\in\labels}\sum_{\ell\in\KK_{i,j}}
  \lambda^{\Bigp{\ell}}_{i,j}H^{\Bigp{\ell}}_{i,j}\Bigp{s,G}
  +\tr\Bigp{GR}.
\]
Evaluating this lifted identity on sampled data recovers the corresponding
identity with $\Ic_f$, $\Pc_f$, and
$H^{\Bigp{\ell}}\Bigp{x_i,x_j}$.
The PSD slack decomposition of $\tr\Bigp{GR}$ gives the residual terms $R_i$.

\section{An SDP formulation for candidate-lemma search}
\label{app:candidate-lemma-sdp-formulation}

This appendix derives the SDP formulation used to search for candidate lemmas.
The formulation is obtained from \eqref{eq:indexed-d-sdp-pep}.

We use the lifted notation $H^{\Bigp{\ell}}_{i,j}\Bigp{s,G}$,
$I\Bigp{s,G}$, and $P\Bigp{s,G}$ from
\Cref{ass:appendix-finite-lifted-representation}. As in \Cref{sec:automatic},
define
\[
  \mathcal J
  \triangleq
  \Bigset{(i,j,\ell): i,j\in\labels,\ 
  \ell\in\KK_{i,j}}.
\]
For $h=(i,j,\ell)\in\mathcal J$, write
\[
  H_h\Bigp{s,G}\triangleq H^{\Bigp{\ell}}_{i,j}\Bigp{s,G},
  \qquad
  q_h\triangleq q^{\Bigp{\ell}}_{i,j},
  \qquad
  Q_h\triangleq Q^{\Bigp{\ell}}_{i,j}.
\]
Let $\mathcal L_{\mathrm{cand}}$ be the candidate-lemma index set. For each
$\kappa\in\mathcal L_{\mathrm{cand}}$, let
$\mathcal I_\kappa\subseteq\mathcal J$ be the hypotheses that may appear in
that candidate lemma.
For fixed coefficients and residual
matrix, a candidate lemma has the lifted analogue of the proof identities in
the main text:
\[
  \Psi_\kappa\Bigp{s,G}
  \triangleq
  \sum_{h\in\mathcal I_\kappa}
  \alpha_h^{\Bigp{\kappa}}H_h\Bigp{s,G}
  +\tr\Bigp{GR_\kappa}
  \geq0,
\]
where
\[
  \alpha_h^{\Bigp{\kappa}}\geq0,\qquad
  R_\kappa\succeq0.
\]
Thus $\Psi_\kappa\geq0$ is not a new assumption: it is certified by the original
hypotheses together with the nonnegative residual term $\tr\Bigp{GR_\kappa}$.
The PSD matrix $R_\kappa$ is the lifted representation of the residual part; a
PSD decomposition of $\tr\Bigp{GR_\kappa}$ gives residual squares of the kind
denoted by $R_i$ in \eqref{eq:generic-proof}.
Equivalently, if
$\Psi_\kappa\Bigp{s,G}=\inner{s}{q_{\Psi_\kappa}}+\tr\Bigp{GQ_{\Psi_\kappa}}$, then
\[
  q_{\Psi_\kappa}
  =
  \sum_{h\in\mathcal I_\kappa}
  \alpha_h^{\Bigp{\kappa}}q_h,
  \qquad
  Q_{\Psi_\kappa}
  =
  \sum_{h\in\mathcal I_\kappa}
  \alpha_h^{\Bigp{\kappa}}Q_h+R_\kappa.
\]

If the candidate lemmas $\Psi_\kappa\geq0$ are fixed and added to
\eqref{eq:indexed-sdp-pep}, the augmented Lagrangian dual has multipliers
$\tau\geq0$, $\lambda_h\geq0$, and $\eta_\kappa\geq0$, with constraints
\[
  \tau q_I-q_P
  -\sum_{h\in\mathcal J}\lambda_hq_h
  -\sum_{\kappa\in\mathcal L_{\mathrm{cand}}}
  \eta_\kappa q_{\Psi_\kappa}
  =0,
\]
\[
  R
  \triangleq
  \tau Q_I-Q_P
  -\sum_{h\in\mathcal J}\lambda_hQ_h
  -\sum_{\kappa\in\mathcal L_{\mathrm{cand}}}
  \eta_\kappa Q_{\Psi_\kappa}
  \succeq0.
\]
Here $R$ denotes the aggregate certificate slack matrix, distinct from the
per-slot residual matrices $R_\kappa$ and $\widehat R_\kappa$.
Searching over candidate lemmas would introduce products
$\eta_\kappa\alpha_h^{\Bigp{\kappa}}$ and $\eta_\kappa R_\kappa$. Absorb
them by setting
\[
  \zeta_h^{\Bigp{\kappa}}
  =\eta_\kappa\alpha_h^{\Bigp{\kappa}},\qquad
  \widehat R_\kappa=\eta_\kappa R_\kappa.
\]
The absorbed variables satisfy
\[
  \zeta_h^{\Bigp{\kappa}}\in\R_+,\qquad
  \widehat R_\kappa\succeq0.
\]
Thus the target-restricted SDP feasible set used for candidate-lemma search,
with the aggregate slack written after absorption, is
\begin{equation}
  \begin{aligned}
    \problemfind
    & \tau,\quad \{\lambda_h\},\quad
    \{\zeta_h^{\Bigp{\kappa}}\},\quad
    \{\widehat R_\kappa\}\\
    \problemsubjectto
    & 0\leq\tau\leq\bar\rho,\\
    & \lambda_h\geq0,\quad
    \forall h\in\mathcal J,\\
    & \zeta_h^{\Bigp{\kappa}}\in\R_+,
    \quad \forall \kappa\in\mathcal L_{\mathrm{cand}},\quad
    h\in\mathcal I_\kappa,\\
    & \widehat R_\kappa\succeq0,
    \quad \forall \kappa\in\mathcal L_{\mathrm{cand}},\\
    &
    \begin{aligned}[t]
      0={}&
      \tau q_I-q_P
      -\sum_{h\in\mathcal J}\lambda_hq_h
      -\sum_{\kappa\in\mathcal L_{\mathrm{cand}}}
      \sum_{h\in\mathcal I_\kappa}
      \zeta_h^{\Bigp{\kappa}}q_h,
    \end{aligned}\\
    &
    \begin{aligned}[t]
      R\triangleq{}&
      \tau Q_I-Q_P
      -\sum_{h\in\mathcal J}\lambda_hQ_h\\
      &-\sum_{\kappa\in\mathcal L_{\mathrm{cand}}}
      \Bigp{
        \sum_{h\in\mathcal I_\kappa}
        \zeta_h^{\Bigp{\kappa}}Q_h
        +\widehat R_\kappa
      }
      \succeq0.
    \end{aligned}
  \end{aligned}
  \tag{CL-SDP}
  \label{eq:candidate-lemma-sdp}
\end{equation}
One can encourage low residual complexity by adding rank-minimization
heuristics on the slack $R$, such as log-det objectives
\citep{fazel2003log}, while keeping the relaxed target $\bar\rho$ fixed.

After solving \eqref{eq:candidate-lemma-sdp}, recover the candidate-lemma
multiplier in slot $\kappa$ by setting
\[
  \eta_\kappa
  \triangleq
  \sum_{h\in\mathcal I_\kappa}
  \zeta_h^{\Bigp{\kappa}}.
\]
If $\eta_\kappa>0$, use the normalization
\[
  \alpha_h^{\Bigp{\kappa}}
  =
  \eta_\kappa^{-1}\zeta_h^{\Bigp{\kappa}},
  \qquad
  R_\kappa
  =
  \eta_\kappa^{-1}\widehat R_\kappa.
\]
With these recovered coefficients, $\Psi_\kappa\Bigp{s,G}$ has the form
displayed above and appears in the proof identity with multiplier $\eta_\kappa$.
If $\eta_\kappa=0$, then
$\zeta_h^{\Bigp{\kappa}}=0$ for every
$h\in\mathcal I_\kappa$, so that slot carries only residual slack.
Define
\[
  \bar R
  \triangleq
  R
  +\sum_{\kappa\in\mathcal L_{\mathrm{cand}}:\ \eta_\kappa=0}
  \widehat R_\kappa
  \succeq0.
\]
Every feasible point of \eqref{eq:candidate-lemma-sdp} yields the proof
identity
\[
  \tau I\Bigp{s,G}-P\Bigp{s,G}
  =
  \sum_{h\in\mathcal J}\lambda_hH_h\Bigp{s,G}
  +\sum_{\kappa\in\mathcal L_{\mathrm{cand}}:\ \eta_\kappa>0}
  \eta_\kappa\Psi_\kappa\Bigp{s,G}
  +\tr\Bigp{G\bar R}
  \geq0.
\]
Expanding each selected $\Psi_\kappa$ recovers a certificate using only the
original hypotheses and PSD residual terms, while keeping the selected
candidate lemmas visible produces the intermediate-lemma proof identity used in
\Cref{sec:automatic}.

\section{Detailed PEP formulations for experimental
examples}\label{app:example-peps}

This section instantiates the sampled-data notation of
\Cref{app:general-pep-framework} for the two examples in the main text and
writes the concrete lifted SDP formulations used in the certificate searches.
Each finite PEP fixes a label set $\labels$, sampled data
$\mathcal D=\Bigset{\Bigp{x_i,f_i,g_i}}_{i\in\labels}$, and concrete indexed
hypotheses. In the examples below, each active ordered pair carries a single
interpolation hypothesis, denoted $H\Bigp{x_i,x_j}\geq0$.

\subsection{One-step gradient descent}\label{app:pep-gd}

Let $N=1$ and $x_1=x_0-\gamma\nabla f\Bigp{x_0}$. The class is
$\F_{\mu,L}$ with $0\leq\mu<L$, the initialization constraint is on
$\Ic_f=f\Bigp{x_0}-f_\star$, and the performance measure is
$\Pc_f=f\Bigp{x_1}-f_\star$.

\paragraph{Functional-residual PEP.}
Set $\labels=\Bigset{0,1,\star}$ and
$\mathcal D=\Bigset{\Bigp{x_i,f_i,g_i}}_{i\in\labels}$, with $g_\star=0$.
The sampled version of \eqref{eq:f-pep} is
\begin{equation}
  \begin{aligned}
    \problemhead{maximize}{\substack{x_0,x_1,x_\star,g_0,g_1\in\R^d,\\
    f_0,f_1,f_\star\in\R,\ d\geq1}}
    &
    f_1-f_\star \\
    \problemsubjectto
    & f_0-f_\star\leq1,\\
    & x_1=x_0-\gamma g_0,\\
    & H_{\mu,L}\Bigp{x_i,x_j}\geq0,
    \qquad i,j\in\Bigset{0,1,\star},\ i\neq j.
  \end{aligned}
  \tag{GD-f-PEP}
  \label{eq:gd-f-pep}
\end{equation}
The omitted self-pairs are only the trivial constraints
$H_{\mu,L}\Bigp{x_i,x_i}\equiv0$. For $i,j\in\labels$,
\[
  H_{\mu,L}\Bigp{x_i,x_j}
  \triangleq
  f_i-f_j-\inner{g_j}{x_i-x_j}
  -\frac{1}{2\Bigp{L-\mu}}
  \Bignorm{g_i-g_j-\mu\Bigp{x_i-x_j}}^2
  -\frac{\mu}{2}\Bignorm{x_i-x_j}^2.
\]

\paragraph{Lifted SDP.}
Introduce
\[
  s=\Bigp{s_0,s_1}=\Bigp{f_0-f_\star,\ f_1-f_\star}.
\]
The Gram matrix is
\begin{equation}
  \label{eq:gd-gram-matrix}
  G=
  \begin{bmatrix}
    \Bignorm{x_0-x_\star}^2 & \inner{g_0}{x_0-x_\star} &
    \inner{g_1}{x_0-x_\star}\\
    \inner{g_0}{x_0-x_\star} & \Bignorm{g_0}^2 & \inner{g_0}{g_1}\\
    \inner{g_1}{x_0-x_\star} & \inner{g_1}{g_0} & \Bignorm{g_1}^2
  \end{bmatrix}
  \in\mathbb S_+^3.
\end{equation}
The update has been substituted through
$x_1-x_\star=\Bigp{x_0-x_\star}-\gamma g_0$. The associated basis selector
vectors in the coordinates of \eqref{eq:gd-gram-matrix} are
\[
  \begin{aligned}
    \widehat{x}_\star&=0,
    &
    \widehat{x}_0&=
    \begin{bmatrix}1\\0\\0
    \end{bmatrix},
    &
    \widehat{x}_1&=
    \begin{bmatrix}1\\-\gamma\\0
    \end{bmatrix},\\
    \widehat{g}_\star&=0,
    &
    \widehat{g}_0&=
    \begin{bmatrix}0\\1\\0
    \end{bmatrix},
    &
    \widehat{g}_1&=
    \begin{bmatrix}0\\0\\1
    \end{bmatrix}.
  \end{aligned}
\]
Here the hats denote coordinate selector vectors in $\R^3$, not new sampled
points; they select $x_i-x_\star$ and $g_i$ in the ordered coordinates
underlying $G$. Let
$e_0,e_1$ be the canonical basis of $\R^2$, set
$q_I=e_0$, $q_P=e_1$, $e_\star=0$, and
$q_{i,j}=e_i-e_j$. Define $Q_{i,j}\in\mathbb S^3$ by
\[
  \tr\Bigp{GQ_{i,j}}
  =
  -\inner{g_j}{x_i-x_j}
  -\frac{1}{2\Bigp{L-\mu}}
  \Bignorm{g_i-g_j-\mu\Bigp{x_i-x_j}}^2
  -\frac{\mu}{2}\Bignorm{x_i-x_j}^2.
\]
Equivalently, the interpolation residuals are represented as
\[
  H_{\mu,L}\Bigp{x_i,x_j}
  =
  \inner{s}{q_{i,j}}+\tr\Bigp{GQ_{i,j}}.
\]
Thus the reduced lifted problem is
\begin{equation}
  \begin{aligned}
    \problemhead{maximize}{s\in\R^2,\ G\in\mathbb S_+^3}
    &
    \inner{s}{q_P} \\
    \problemsubjectto
    & \inner{s}{q_I}\leq1,\\
    & \inner{s}{q_{i,j}}+\tr\Bigp{GQ_{i,j}}\geq0,
    \qquad i,j\in\Bigset{0,1,\star},\ i\neq j.
  \end{aligned}
  \tag{GD-SDP-PEP}
  \label{eq:gd-sdp-pep}
\end{equation}

\paragraph{Lagrangian dual certificate.}
The reduced Lagrangian dual is
\begin{equation}
  \begin{aligned}
    \problemhead{minimize}{\substack{\tau\geq0,\\
    \lambda_{i,j}\geq0,\ i,j\in\Bigset{0,1,\star},\ i\neq j}}
    &
    \tau \\
    \problemsubjectto
    & \tau q_I-q_P-
    \displaystyle\sum_{\substack{i,j\in\Bigset{0,1,\star}\\i\neq j}}
    \lambda_{i,j}q_{i,j}=0,\\
    & -\displaystyle\sum_{\substack{i,j\in\Bigset{0,1,\star}\\i\neq j}}
    \lambda_{i,j}Q_{i,j}\succeq0.
  \end{aligned}
  \tag{GD-D-SDP-PEP}
  \label{eq:gd-d-sdp-pep}
\end{equation}

\subsection{Candidate-lemma SDP for GD}
\label{app:gd-candidate-lemma-sdp}

The generic candidate-lemma SDP in
\Cref{app:candidate-lemma-sdp-formulation} describes a large family of possible
derived inequalities. To obtain a simple and interpretable GD proof,
we constrain the space of allowable lemmas before solving the SDP.
Begin by fixing $\tau=\bar\rho=\rho_\gamma$ from
\eqref{eq:gd-certificate-ingredients}. We also fix the multiplier support to
$(\star,0)$, $(\star,1)$, and $(0,1)$, as discovered using exhaustive
sparsification, and fix their values to the same closed-form multipliers used
in the classical proof for GD, given in \eqref{eq:gd-certificate-ingredients}.

We now restrict the lemmas to the space of
interpolation constraints
over \emph{weaker} regularity parameters $0\leq a<b$, such that $a
\leq \mu$ and $L \leq b$. For a pair of sampled points, write the interpolation
inequality with lower curvature $a$ and upper curvature $b$ as
\[
  \begin{aligned}
    H_{a,b}\Bigp{x_i,x_j}
    &=
    f_i-f_j-\inner{g_j}{x_i-x_j}
    -\frac{a}{2}\Bignorm{x_i-x_j}^2\\
    &\quad
    -\frac{1}{2\Bigp{b-a}}
    \Bignorm{g_i-g_j-a\Bigp{x_i-x_j}}^2
    \geq0.
  \end{aligned}
\]
Grouping by Gram terms, we get the following expression
\begin{equation}
  \begin{aligned}
    H_{a,b}\Bigp{x_i,x_j}
    &=
    f_i-f_j-\inner{g_j}{x_i-x_j}
    -\frac{ab}{2\Bigp{b-a}}\Bignorm{x_i-x_j}^2\\
    &\quad
    +\frac{a}{b-a}\inner{x_i-x_j}{g_i-g_j}
    -\frac{1}{2\Bigp{b-a}}\Bignorm{g_i-g_j}^2.
  \end{aligned}
  \label{eq:gd-ab-interpolation-expanded}
\end{equation}
Let $\operatorname{sym}(M)=\frac12\Bigp{M+M^\top}$, and let
$Q_{i,j}^{a,b}$ denote the matrix satisfying
\[
  H_{a,b}\Bigp{x_i,x_j}
  =
  \inner{s}{q_{i,j}}+\tr\Bigp{G Q_{i,j}^{a,b}}.
\]
Using the basis selector vectors from \eqref{eq:gd-gram-matrix},
\[
  \begin{aligned}
    Q_{i,j}^{a,b}
    &=
    -\operatorname{sym}\Bigp{\widehat{g}_j
    \Bigp{\widehat{x}_i-\widehat{x}_j}^\top}
    -\frac{ab}{2\Bigp{b-a}}
    \Bigp{\widehat{x}_i-\widehat{x}_j}
    \Bigp{\widehat{x}_i-\widehat{x}_j}^\top\\
    &\quad
    +\frac{a}{b-a}
    \operatorname{sym}\Bigp{\Bigp{\widehat{x}_i-\widehat{x}_j}
    \Bigp{\widehat{g}_i-\widehat{g}_j}^\top}
    -\frac{1}{2\Bigp{b-a}}
    \Bigp{\widehat{g}_i-\widehat{g}_j}
    \Bigp{\widehat{g}_i-\widehat{g}_j}^\top.
  \end{aligned}
\]
This isolates how the regularity parameters affect the terms
of our Gram matrix. We search over them by writing
\[
  \begin{aligned}
    \widehat Q_{i,j}(c)
    &=
    -\operatorname{sym}\Bigp{\widehat{g}_j
    \Bigp{\widehat{x}_i-\widehat{x}_j}^\top}
    -c_{xx}\Bigp{\widehat{x}_i-\widehat{x}_j}
    \Bigp{\widehat{x}_i-\widehat{x}_j}^\top\\
    &\quad
    -c_{xg}\operatorname{sym}\Bigp{\Bigp{\widehat{x}_i-\widehat{x}_j}
    \Bigp{\widehat{g}_i-\widehat{g}_j}^\top}
    -c_{gg}\Bigp{\widehat{g}_i-\widehat{g}_j}
    \Bigp{\widehat{g}_i-\widehat{g}_j}^\top,
  \end{aligned}
\]
where $c\triangleq(c_{xx},c_{xg},c_{gg})$.
To extract the values of the regularity parameters, compare the
expression for $\widehat Q_{i,j}(c)$ with $Q_{i,j}^{a,b}$. If $c_{gg}>0$, the
following identities allow us to extract them:
\[
  a=-\frac{c_{xg}}{2c_{gg}},
  \qquad
  b=a+\frac{1}{2c_{gg}}.
\]
The remaining coefficient satisfies
$c_{xx}=ab\,c_{gg}$ when the candidate is exactly an interpolation inequality
for $\F_{a,b}$.
Let $Q_{i,j}^{\mu,L}$ be the corresponding Gram matrix for the original class
$\F_{\mu,L}$. The following SDP searches over the coefficients in $c$ with
these multipliers held fixed:
\begin{equation}
  \begin{aligned}
    \problemhead{minimize}{c_{xx},c_{xg},c_{gg}}
    &
    \tr\Bigp{R(c)}\\
    \problemsubjectto
    & c_{xx}\geq0,\qquad c_{xg}\leq0,\qquad c_{gg}\geq0,\\
    & \widehat Q_{i,j}(c)-Q_{i,j}^{\mu,L}\succeq0,
    \qquad
    (i,j)\in
    \Bigset{\Bigp{\star,0},\Bigp{\star,1},\Bigp{0,1}},\\
    & R(c)\triangleq
    -\lambda_{\star,0}\widehat Q_{\star,0}(c)
    -\lambda_{\star,1}\widehat Q_{\star,1}(c)
    -\lambda_{0,1}\widehat Q_{0,1}(c)
    \succeq0.
  \end{aligned}
  \tag{GD-CL-SDP}
  \label{eq:gd-candidate-lemma-sdp}
\end{equation}
For a coefficient vector satisfying this relation, the LMI
$\widehat Q_{i,j}(c)-Q_{i,j}^{\mu,L}\succeq0$ certifies
validity by enforcing $H_{a,b}\Bigp{x_i,x_j}\geq
H_{\mu,L}\Bigp{x_i,x_j} \geq 0$ in primal space.
The last line is
the residual slack in the certificate, and minimizing
$\tr\Bigp{R(c)}$ heuristically selects a
low-rank residual.
The SDP solutions in \Cref{fig:gd-fitted-curvature-scan}
do indeed form valid interpolation constraints, and the extracted
fitted constants
in \eqref{eq:gd-fitted-curvatures} are
$\mu_\gamma=a$ and $L_\gamma=b$.

\subsection{Smooth-convex fast gradient method}\label{app:pep-fgm-cvx}

For a horizon $N$, consider the smooth-convex FGM recurrence
\[
  y_0=x_0,\qquad
  x_{k+1}=y_k-\frac1L\nabla f\Bigp{y_k},\qquad
  y_{k+1}=x_{k+1}+\frac{k}{k+3}\Bigp{x_{k+1}-x_k},
\]
with the $x$-update imposed for $k\in\llbracket0,N-1\rrbracket$ and the
$y$-update imposed for $k\in\llbracket0,N-2\rrbracket$. The class is $\F_L$,
the initialization constraint is $\Ic_f=\Bignorm{x_0-x_\star}^2$, and
the performance measure is
$\Pc_f=f\Bigp{x_N}-f_\star$.

\paragraph{Function-value PEP.}
Set
\[
  \labels_N=\llbracket 0,N \rrbracket\cup\Bigset{\star},
  \qquad
  \bar x_i=y_i\quad i\in\llbracket 0,N-1 \rrbracket,\qquad
  \bar x_N=x_N,\qquad
  \bar x_\star=x_\star.
\]
The sampled oracle data are
\[
  \mathcal D_N
  =
  \Bigset{\Bigp{\bar x_i,f_i,g_i}}_{i\in\labels_N},
  \qquad
  f_i=f\Bigp{\bar x_i},\quad g_i=\nabla f\Bigp{\bar x_i},
  \quad g_\star=0.
\]
Then, writing $x_{0:N}=\Bigp{x_0,\dots,x_N}$ and
$y_{0:N-1}=\Bigp{y_0,\dots,y_{N-1}}$,
\begin{equation}
  \begin{aligned}
    \problemhead{maximize}{\mathcal D_N,\ x_{0:N},\ y_{0:N-1},
    x_\star,\ d\geq1}
    &
    f_N-f_\star \\
    \problemsubjectto
    & \Bignorm{x_0-x_\star}^2\leq1,\\
    & y_0=x_0,\\
    & x_{k+1}=y_k-\frac1L g_k,
    \qquad k\in\llbracket 0,N-1 \rrbracket,\\
    & y_k=x_k+\frac{k-1}{k+2}\Bigp{x_k-x_{k-1}},
    \qquad k\in\llbracket 1,N-1 \rrbracket,\\
    & H_L\Bigp{\bar x_i,\bar x_j}\geq0,
    \qquad i,j\in\labels_N,\ i\neq j.
  \end{aligned}
  \tag{FGM-f-PEP}
  \label{eq:fgm-cvx-f-pep}
\end{equation}
Here the concrete smooth-convex interpolation hypothesis is the main-text
inequality
\[
  H_L\Bigp{\bar x_i,\bar x_j}
  \triangleq
  f_i-f_j-\inner{g_j}{\bar x_i-\bar x_j}
  -\frac{1}{2L}\Bignorm{g_i-g_j}^2.
\]

\paragraph{Lifted SDP.}
Use label-based scalar coordinates
\[
  s=\Bigp{f_0-f_\star,\dots,f_N-f_\star}\in\R^{N+1},
\]
and let $\operatorname{Gram}\Bigp{v_1,\ldots,v_k}$ denote the matrix with
entries $\inner{v_i}{v_j}$. Set
\[
  G=
  \operatorname{Gram}\Bigp{x_0-x_\star,\dots,x_N-x_\star,
  y_0-x_\star,\dots,y_{N-1}-x_\star,\ g_0,\dots,g_N}
  \in\mathbb S_+^{3N+2}.
\]
Let $e_0,\dots,e_N$ be the canonical basis of $\R^{N+1}$, set
$e_\star=0$, $q_P=e_N$, $q_I=0$, $Q_I=E_0$, and
$q_{i\to j}=e_i-e_j$. We write the smooth-convex interpolation inequalities as
the concrete linear forms
\[
  H_L\Bigp{\bar x_i,\bar x_j}
  =
  \inner{s}{q_{i\to j}}+\tr\Bigp{GQ_{i\to j}}
  \geq0,\qquad i,j\in\labels_N,\ i\neq j,
\]
where $Q_{i\to j}\in\mathbb S^{3N+2}$ is chosen so that
\[
  \tr\Bigp{GQ_{i\to j}}
  =
  -\inner{g_j}{\bar x_i-\bar x_j}
  -\frac{1}{2L}\Bignorm{g_i-g_j}^2.
\]
The lifted problem is
\begin{equation}
  \begin{aligned}
    \problemhead{maximize}{s\in\R^{N+1},\ G\in\mathbb S_+^{3N+2}}
    &
    \inner{s}{q_P} \\
    \problemsubjectto
    & \tr\Bigp{GQ_I}\leq1,\\
    & \tr\Bigp{M_kG}=0,\qquad \forall k\in\llbracket 0,N-1 \rrbracket,\\
    & \tr\Bigp{N_0G}=0,\\
    & \tr\Bigp{N_kG}=0,\qquad \forall k\in\llbracket 1,N-1 \rrbracket,\\
    & H_L\Bigp{\bar x_i,\bar x_j}\geq0,
    \qquad \forall i,j\in\labels_N:\ i\neq j.
  \end{aligned}
  \tag{FGM-SDP-PEP}
  \label{eq:fgm-cvx-sdp-pep}
\end{equation}
The coefficient matrices encode
\[
  \tr\Bigp{GQ_I}=\Bignorm{x_0-x_\star}^2,\qquad
  \tr\Bigp{M_kG}=\Bignorm{x_{k+1}-y_k+\frac1L g_k}^2,
\]
\[
  \tr\Bigp{N_0G}=\Bignorm{y_0-x_0}^2,
  \qquad
  \tr\Bigp{N_kG}=
  \Bignorm{y_k-x_k-\frac{k-1}{k+2}\Bigp{x_k-x_{k-1}}}^2,
\]
while the matrices $Q_{i\to j}$ encode the Gram part of the interpolation
hypotheses.

\paragraph{Lagrangian dual certificate.}
The Lagrangian dual, with free multipliers for the method equalities, is
\begin{equation}
  \begin{aligned}
    \problemhead{minimize}{\substack{\tau\geq0,\ \eta_k\in\R,\ k\in\llbracket
        0,N-1 \rrbracket,\\
        \zeta_k\in\R,\ k\in\llbracket 0,N-1 \rrbracket,\\
    \lambda_{i\to j}\geq0,\ i,j\in\labels_N,\ i\neq j}}
    &
    \tau \\
    \problemsubjectto
    & \tau q_I-q_P-
    \displaystyle\sum_{\substack{i,j\in\labels_N\\i\neq j}}
    \lambda_{i\to j}q_{i\to j}=0,\\
    &
    \begin{aligned}[t]
      \tau Q_I
      &+\displaystyle\sum_{k\in\llbracket 0,N-1 \rrbracket}\eta_kM_k
      +\displaystyle\sum_{k\in\llbracket 0,N-1 \rrbracket}\zeta_kN_k\\
      &-\displaystyle\sum_{\substack{i,j\in\labels_N\\i\neq j}}
      \lambda_{i\to j}Q_{i\to j}\succeq0.
    \end{aligned}
  \end{aligned}
  \tag{FGM-D-SDP-PEP}
  \label{eq:fgm-cvx-d-sdp-pep}
\end{equation}
The accompanying FGM example code eliminates the recurrence equalities before
solving; this produces an equivalent reduced Lagrangian dual with the
same interpolation
multiplier labels $\lambda_{i\to j}$.

\section{Original-curvature residuals for gradient descent}
\label{app:gd-original-residuals}

The comparison in \Cref{rem:gd-original-curvatures} uses the residual
obtained by specializing the branchwise slack decomposition of
\citet[Theorem~3.3]{taylor2018proximal} to $h\equiv0$. Let
\[
  \begin{gathered}
    \alpha_\gamma^-
    \triangleq
    2L\Bigp{2-\gamma\mu}
    -\mu\Bigp{2-\gamma\mu}^2
    -\gamma^2L^2\mu,
    \qquad
    \beta_\gamma^-\triangleq2-\gamma\Bigp{L+\mu},
    \\
    \alpha_\gamma^+
    \triangleq
    \gamma L\Bigp{L^2+\mu^2}
    -2\Bigp{L^2-L\mu+\mu^2},
    \qquad
    \beta_\gamma^+\triangleq\gamma\Bigp{L+\mu}-2.
  \end{gathered}
\]
Then
\begin{equation}
  R_\gamma'
  \triangleq
  \begin{cases}
    R_\gamma^-,
    & 0<\gamma\leq\frac{2}{L+\mu},\\[1mm]
    R_\gamma^+,
    & \frac{2}{L+\mu}\leq\gamma\leq\frac{2}{L},
  \end{cases}
  \label{eq:gd-original-residuals}
\end{equation}
where
\[
  \begin{aligned}
    R_\gamma^-
    &\triangleq
    \frac{\Bigp{2-\gamma\mu}\beta_\gamma^-}{2\alpha_\gamma^-}
    \Bignorm{
      \Bigp{1-\gamma\mu}\nabla f\Bigp{x_0}-\nabla f\Bigp{x_1}
    }^2
    \\
    &\quad
    +\frac{\gamma L\mu^2\Bigp{2-\gamma\mu}}{2\Bigp{L-\mu}}
    \Bignorm{
      x_0-x_\star
      -\frac{\nabla f\Bigp{x_0}+\nabla f\Bigp{x_1}}
      {\mu\Bigp{2-\gamma\mu}}
    }^2
    \\
    &\quad
    +\frac{\gamma\mu\alpha_\gamma^-}
    {2L\Bigp{L-\mu}\Bigp{2-\gamma\mu}}
    \Bignorm{
      \frac{\Bigp{\gamma\mu-1}L\beta_\gamma^-}{\alpha_\gamma^-}
      \nabla f\Bigp{x_0}
      +\frac{L\beta_\gamma^-}{\alpha_\gamma^-}
      \nabla f\Bigp{x_1}
    }^2,
    \\
    R_\gamma^+
    &\triangleq
    \frac{\Bigp{2-\gamma L}\beta_\gamma^+}{2\gamma\alpha_\gamma^+}
    \Bignorm{
      \Bigp{1-\gamma L}\nabla f\Bigp{x_0}-\nabla f\Bigp{x_1}
    }^2
    \\
    &\quad
    +\frac{\gamma L^2\mu\Bigp{2-\gamma L}}{2\Bigp{L-\mu}}
    \Bignorm{
      x_0-x_\star
      +\frac{1-\gamma L-\gamma\mu}{\gamma L\mu}
      \nabla f\Bigp{x_0}
      -\frac{1}{\gamma L\mu}\nabla f\Bigp{x_1}
    }^2
    \\
    &\quad
    +\frac{\gamma\alpha_\gamma^+}{2\mu\Bigp{L-\mu}}
    \Bignorm{
      \frac{\Bigp{\gamma L-1}L\beta_\gamma^+}{\gamma\alpha_\gamma^+}
      \nabla f\Bigp{x_0}
      +\frac{L\beta_\gamma^+}{\gamma\alpha_\gamma^+}
      \nabla f\Bigp{x_1}
    }^2 .
  \end{aligned}
\]

\section{Proximal certificate proofs}
\label{app:proximal-certificate-proofs}

\subsection{Proximal point residual certificate}
\label{app:ppa-residual-proof}

\pparesidualvaluethm*

\begin{proof}[Proof of \Cref{thm:ppa-residual-value}]
  Set $A_0=0$ and $A_k=\sum_{i=1}^k\alpha_i$. For sampled points define the
  convexity gaps
  \[
    C_{i,j}
    \triangleq
    F\Bigp{x_i}-F\Bigp{x_j}-\inner{g_j}{x_i-x_j},
    \qquad
    C_{\star,j}
    \triangleq
    F_\star-F\Bigp{x_j}-\inner{g_j}{x_\star-x_j}.
  \]
  Since $g_j\in\partial F\Bigp{x_j}$, all these quantities are
  nonnegative. Define
  \[
    \mathcal V_k
    \triangleq
    \Bignorm{x_k-x_\star}_B^2
    +2A_k\Bigp{F\Bigp{x_k}-F_\star}
    +A_k^2\Bignorm{g_k}_{B^{-1}}^2 .
  \]
  Using $x_0-x_1=\alpha_1B^{-1}g_1$,
  \[
    \Bignorm{x_0-x_\star}_B^2-\mathcal V_1
    =
    2\alpha_1
    \Bigp{\inner{g_1}{x_1-x_\star}-F\Bigp{x_1}+F_\star}
    =
    2\alpha_1 C_{\star,1}
    \geq0.
  \]
  Thus $\mathcal V_1\leq\Bignorm{x_0-x_\star}_B^2$.

  For the decrement, fix $k\in\llbracket1,N-1\rrbracket$ and set
  $A=A_k$, $a=\alpha_{k+1}$, and $h=g_{k+1}$. From
  $x_k-x_{k+1}=aB^{-1}h$,
  \[
    \Bignorm{x_k-x_\star}_B^2-\Bignorm{x_{k+1}-x_\star}_B^2
    =
    2a\inner{h}{x_{k+1}-x_\star}
    +a^2\Bignorm{h}_{B^{-1}}^2
  \]
  and
  \[
    F\Bigp{x_k}-F\Bigp{x_{k+1}}
    =
    C_{k,k+1}+a\Bignorm{h}_{B^{-1}}^2 .
  \]
  Substitution into $\mathcal V_k-\mathcal V_{k+1}$ gives
  \[
    \mathcal V_k-\mathcal V_{k+1}
    =
    2aC_{\star,k+1}
    +2AC_{k,k+1}
    +A^2\Bigp{\Bignorm{g_k}_{B^{-1}}^2-\Bignorm{g_{k+1}}_{B^{-1}}^2}.
  \]
  The adjacent gaps satisfy
  \[
    C_{k,k+1}+C_{k+1,k}
    =
    \inner{g_k-g_{k+1}}{x_k-x_{k+1}}
    =
    a\inner{g_k-g_{k+1}}{B^{-1}g_{k+1}},
  \]
  so
  \[
    \Bignorm{g_k}_{B^{-1}}^2-\Bignorm{g_{k+1}}_{B^{-1}}^2
    =
    \frac{2}{a}\Bigp{C_{k,k+1}+C_{k+1,k}}
    +\Bignorm{g_k-g_{k+1}}_{B^{-1}}^2 .
  \]
  Therefore
  \[
    \begin{aligned}
      \mathcal V_k-\mathcal V_{k+1}
      ={}&
      2aC_{\star,k+1}
      +\Bigp{2A+\frac{2A^2}{a}}C_{k,k+1}
      +\frac{2A^2}{a}C_{k+1,k} \\
      &\quad
      +A^2\Bignorm{g_k-g_{k+1}}_{B^{-1}}^2
      \geq0.
    \end{aligned}
  \]
  Consequently
  $\mathcal V_N\leq\mathcal V_1\leq\Bignorm{x_0-x_\star}_B^2\leq\Delta_0^2$. Since
  $F\Bigp{x_N}-F_\star\geq0$, this implies
  \[
    A_N^2\Bignorm{g_N}_{B^{-1}}^2
    \leq
    \mathcal V_N
    \leq
    \Delta_0^2,
  \]
  and hence $\Bignorm{g_N}_{B^{-1}}\leq\Delta_0/A_N$.
  Finally, convexity gives
  $F\Bigp{x_N}-F_\star\leq\inner{g_N}{x_N-x_\star}$, and Young's
  inequality in the
  $B/B^{-1}$ pair yields
  \[
    2A_N\Bigp{F\Bigp{x_N}-F_\star}
    \leq
    \Bignorm{x_N-x_\star}_B^2
    +A_N^2\Bignorm{g_N}_{B^{-1}}^2 .
  \]
  Adding another copy of $2A_N\Bigp{F\Bigp{x_N}-F_\star}$ to both sides gives
  $4A_N\Bigp{F\Bigp{x_N}-F_\star}\leq\mathcal V_N\leq\Delta_0^2$, proving the value
  bound.
\end{proof}

\subsection{Accelerated proximal point saddle-gap certificate}
\label{app:appm-saddle-gap-proof}

\appmsaddlegapthm*

We first prove an operator estimate. Let $\HH$ be a real Hilbert space, let
$\mathcal M:\HH\rightrightarrows\HH$ be maximally monotone, let
$x_\star\in\operatorname{zer}\mathcal M$, and fix $\beta>0$. Starting from
$x_0=y_0=y_{-1}$, define
\[
  x_{k+1}=\Bigp{\operatorname{Id}+\beta\mathcal M}^{-1}\Bigp{y_k},
  \qquad
  y_{k+1}=x_{k+1}
  +\frac{k}{k+2}\Bigp{x_{k+1}-x_k}
  -\frac{k}{k+2}\Bigp{x_k-y_{k-1}}.
\]
For $k\geq1$, set
\[
  r_k\triangleq y_{k-1}-x_k,
  \qquad
  q_k\triangleq\beta^{-1}r_k .
\]
The resolvent equation gives $q_k\in\mathcal M x_k$.

\begin{lemma}[APPM operator estimate]
  \label{lem:appm-operator-estimate}
  For every $N\geq1$,
  \[
    \inner{x_N-x_\star}{y_{N-1}-x_N}
    \leq
    \frac{\Bignorm{y_0-x_\star}^2}{4N}.
  \]
\end{lemma}

\begin{proof}
  The recurrence implies, by induction, the trajectory identities
  \[
    y_k
    =
    y_0-\frac{2}{k+1}\sum_{\ell=1}^{k}\ell r_\ell,
    \qquad
    x_k
    =
    y_0-\frac{2}{k}\sum_{\ell=1}^{k-1}\ell r_\ell-r_k,
    \qquad k\geq1.
  \]
  Set $z=y_0-x_\star$, $A_k=\sum_{\ell=1}^{k}\ell r_\ell$ with $A_0=0$, and
  $d_k=x_k-x_\star$. Define
  \[
    \mathcal V_k
    \triangleq
    k\inner{d_k}{r_k}
    +
    \Bignorm{k r_k-\frac12z}^2 .
  \]
  Using the trajectory formula, $d_k=z-\frac{2}{k}A_{k-1}-r_k$, and hence
  \[
    \mathcal V_k
    =
    \frac14\Bignorm{z}^2
    -2\inner{A_{k-1}}{r_k}
    +k(k-1)\Bignorm{r_k}^2 .
  \]
  A direct subtraction, using $A_{k-1}=A_{k-2}+(k-1)r_{k-1}$ and the same
  trajectory formula for $d_{k-1}-d_k$, gives, for $k\geq2$,
  \[
    k(k-1)\Bigp{d_{k-1}-d_k}
    =
    k(k-1)\Bigp{r_{k-1}+r_k}-2A_{k-1}.
  \]
  Taking the inner product with $r_{k-1}-r_k$ gives
  \[
    k(k-1)\inner{d_{k-1}-d_k}{r_{k-1}-r_k}
    =
    F_k-F_{k-1},
    \qquad
    F_k\triangleq2\inner{A_{k-1}}{r_k}-k(k-1)\Bignorm{r_k}^2 .
  \]
  Since $\mathcal V_k=\frac14\Bignorm{z}^2-F_k$, we obtain
  \[
    \mathcal V_k-\mathcal V_{k-1}
    =
    -k(k-1)\inner{x_{k-1}-x_k}{r_{k-1}-r_k}.
  \]
  Since $r_\ell=\beta q_\ell$ with $q_\ell\in\mathcal M x_\ell$,
  \[
    \inner{x_{k-1}-x_k}{r_{k-1}-r_k}
    =
    \beta\inner{x_{k-1}-x_k}{q_{k-1}-q_k}
    \geq0
  \]
  by monotonicity. Thus $\mathcal V_k\leq\mathcal V_{k-1}$. At $k=1$,
  $x_1=y_0-r_1$, so
  \[
    \mathcal V_1
    =
    \inner{y_0-x_\star-r_1}{r_1}
    +
    \Bignorm{r_1-\frac12\Bigp{y_0-x_\star}}^2
    =
    \frac14\Bignorm{y_0-x_\star}^2 .
  \]
  Therefore $\mathcal V_N\leq\Bignorm{y_0-x_\star}^2/4$. Dropping the
  square term
  in $\mathcal V_N$ gives
  \[
    N\inner{x_N-x_\star}{r_N}
    \leq
    \mathcal V_N
    \leq
    \frac14\Bignorm{y_0-x_\star}^2 .
  \]
  Since $r_N=y_{N-1}-x_N$, the claim follows.
\end{proof}

\begin{proof}[Proof of \Cref{thm:appm-saddle-gap}]
  Apply \Cref{lem:appm-operator-estimate} on the product Hilbert space
  $\HH_1\times\HH_2$ with $x_\star=\Bigp{u_\star,v_\star}$. Write
  $q_N=\Bigp{p_N,s_N}\in\mathcal M\Bigp{u_N,v_N}$. Then
  $p_N\in\partial\Phi\Bigp{\cdot,v_N}\Bigp{u_N}$ and
  $s_N\in\partial\Bigp{-\Phi\Bigp{u_N,\cdot}}\Bigp{v_N}$. Convexity
  in $u$ and convexity of
  $-\Phi\Bigp{u_N,\cdot}$ give
  \[
    \Phi\Bigp{u_N,v_N}-\Phi\Bigp{u_\star,v_N}
    \leq
    \inner{p_N}{u_N-u_\star},
    \qquad
    \Phi\Bigp{u_N,v_\star}-\Phi\Bigp{u_N,v_N}
    \leq
    \inner{s_N}{v_N-v_\star}.
  \]
  Adding and using $r_N=\beta q_N$ yields
  \[
    \Phi\Bigp{u_N,v_\star}-\Phi\Bigp{u_\star,v_N}
    \leq
    \inner{x_N-x_\star}{q_N}
    =
    \beta^{-1}\inner{x_N-x_\star}{r_N}.
  \]
  The operator estimate gives
  \[
    \Phi\Bigp{u_N,v_\star}-\Phi\Bigp{u_\star,v_N}
    \leq
    \frac{\Bignorm{y_0-x_\star}^2}{4\beta N}
    =
    \frac{\Bignorm{u_0-u_\star}^2+\Bignorm{v_0-v_\star}^2}{4\beta N},
  \]
  as claimed.
\end{proof}

\section{FGM active multiplier patterns}\label{app:fgm-multiplier-support}

For the FGM sparsity comparison, the conjectured $2N$ active-multiplier pattern
for the interpolation constraints associated with the function class $\F_L$ is
\[
  \mathcal S_{2N}
  \triangleq
  \Bigset{x_\star\to y_k:\ k\in\llbracket 0,N-1 \rrbracket}
  \cup
  \Bigset{y_k\to y_{k+1}:\ k\in\llbracket 0,N-2 \rrbracket}
  \cup
  \Bigset{y_{N-1}\to x_N}.
\]
Here $u\to v$ denotes the $\F_L$ interpolation constraint with source sample $u$
and target sample $v$. In the FGM label convention of \Cref{app:pep-fgm-cvx},
these three blocks correspond respectively to the active Lagrangian
dual multipliers
$\lambda_{\star\to k}$, $\lambda_{k\to k+1}$, and
$\lambda_{N-1\to N}$. Thus the middle block pairs consecutive momentum points
$y_k$ and $y_{k+1}$, while the only edge involving the final endpoint is the
terminal pair $y_{N-1}\to x_N$. This pattern has size
$N+\Bigp{N-1}+1=2N$.
For instance, when $N=3$ the pattern is
\[
  \Bigset{x_\star\to y_0,\ x_\star\to y_1,\ x_\star\to y_2,
  y_0\to y_1,\ y_1\to y_2,\ y_2\to x_3}.
\]

The following spy-style table records the active interpolation-multiplier
patterns used in the $N=3$ FGM sparsity comparison. Rows are source points and
columns are target points: a bullet in row $u$ and column $v$ means that the
corresponding multiplier $\lambda_{u\to v}$ is above the numerical active-set
threshold, while an empty cell means that it is inactive. The final panel is the
conjectured $2N$ pattern above, which matches the exhaustive active pattern for
this instance. The table is generated from the reduced FGM formulation
equivalent to \eqref{eq:fgm-cvx-d-sdp-pep}.

\input{tables/fgm_multiplier_support_tables.tex}

%% file: tables/fgm_multiplier_support_tables.tex
\begin{table}[H]
\centering
\normalsize
\renewcommand{\arraystretch}{1.15}
\setlength{\tabcolsep}{4.0pt}
\caption{Active interpolation-multiplier patterns for the FGM $N=3$ example. Rows are sources and columns are targets; a bullet marks $\lambda_{i\to j}>10^{-7}$ and dashes mark self-pairs.}
\label{tab:fgm-active-multiplier-patterns}
\noindent\makebox[\textwidth][c]{%
\begin{tabular}[t]{@{}lccccc@{}}
\multicolumn{6}{c}{Raw (total 16)} \\
\toprule
 & $x_\star$ & $y_{0}$ & $y_{1}$ & $y_{2}$ & $x_{3}$ \\
\midrule
$x_\star$ & -- & $\bullet$ & $\bullet$ & $\bullet$ & $\bullet$ \\
$y_{0}$ &  & -- & $\bullet$ & $\bullet$ & $\bullet$ \\
$y_{1}$ &  & $\bullet$ & -- & $\bullet$ & $\bullet$ \\
$y_{2}$ &  & $\bullet$ & $\bullet$ & -- & $\bullet$ \\
$x_{3}$ &  & $\bullet$ & $\bullet$ & $\bullet$ & -- \\
\bottomrule
\end{tabular}\hspace{0.6em}\begin{tabular}[t]{@{}lccccc@{}}
\multicolumn{6}{c}{Plain $\ell_1$ (total 8)} \\
\toprule
 & $x_\star$ & $y_{0}$ & $y_{1}$ & $y_{2}$ & $x_{3}$ \\
\midrule
$x_\star$ & -- & $\bullet$ & $\bullet$ & $\bullet$ & $\bullet$ \\
$y_{0}$ &  & -- &  & $\bullet$ & $\bullet$ \\
$y_{1}$ &  &  & -- &  & $\bullet$ \\
$y_{2}$ &  &  &  & -- & $\bullet$ \\
$x_{3}$ &  &  &  &  & -- \\
\bottomrule
\end{tabular}
}
\par\vspace{1.6em}
\noindent\makebox[\textwidth][c]{%
\begin{tabular}[t]{@{}lccccc@{}}
\multicolumn{6}{c}{Log-sum (total 8)} \\
\toprule
 & $x_\star$ & $y_{0}$ & $y_{1}$ & $y_{2}$ & $x_{3}$ \\
\midrule
$x_\star$ & -- & $\bullet$ & $\bullet$ & $\bullet$ & $\bullet$ \\
$y_{0}$ &  & -- &  & $\bullet$ & $\bullet$ \\
$y_{1}$ &  &  & -- &  & $\bullet$ \\
$y_{2}$ &  &  &  & -- & $\bullet$ \\
$x_{3}$ &  &  &  &  & -- \\
\bottomrule
\end{tabular}\hspace{0.6em}\begin{tabular}[t]{@{}lccccc@{}}
\multicolumn{6}{c}{Norm. log-sum (total 7)} \\
\toprule
 & $x_\star$ & $y_{0}$ & $y_{1}$ & $y_{2}$ & $x_{3}$ \\
\midrule
$x_\star$ & -- & $\bullet$ & $\bullet$ & $\bullet$ & $\bullet$ \\
$y_{0}$ &  & -- & $\bullet$ &  &  \\
$y_{1}$ &  &  & -- & $\bullet$ &  \\
$y_{2}$ &  &  &  & -- & $\bullet$ \\
$x_{3}$ &  &  &  &  & -- \\
\bottomrule
\end{tabular}
}
\par\vspace{1.6em}
\noindent\makebox[\textwidth][c]{%
\begin{tabular}[t]{@{}lccccc@{}}
\multicolumn{6}{c}{Capped $\ell_1$ (total 7)} \\
\toprule
 & $x_\star$ & $y_{0}$ & $y_{1}$ & $y_{2}$ & $x_{3}$ \\
\midrule
$x_\star$ & -- & $\bullet$ & $\bullet$ & $\bullet$ & $\bullet$ \\
$y_{0}$ &  & -- & $\bullet$ &  &  \\
$y_{1}$ &  &  & -- & $\bullet$ &  \\
$y_{2}$ &  &  &  & -- & $\bullet$ \\
$x_{3}$ &  &  &  &  & -- \\
\bottomrule
\end{tabular}\hspace{0.6em}\begin{tabular}[t]{@{}lccccc@{}}
\multicolumn{6}{c}{Conjecture (total 6)} \\
\toprule
 & $x_\star$ & $y_{0}$ & $y_{1}$ & $y_{2}$ & $x_{3}$ \\
\midrule
$x_\star$ & -- & $\bullet$ & $\bullet$ & $\bullet$ &  \\
$y_{0}$ &  & -- & $\bullet$ &  &  \\
$y_{1}$ &  &  & -- & $\bullet$ &  \\
$y_{2}$ &  &  &  & -- & $\bullet$ \\
$x_{3}$ &  &  &  &  & -- \\
\bottomrule
\end{tabular}
}
\end{table}

%% file: bib.bib
@phdthesis{taylorconvex,
  author = {Taylor, Adrien},
  title = {Convex interpolation and performance estimation of first-order methods for convex optimization},
  school = {Universit{\'e} {C}atholique de Louvain},
  year = {2017},
  url = {https://hdl.handle.net/2078.5/67094},
}

@inproceedings{goujaud2023fundamental,
  author = {Goujaud, Baptiste and Dieuleveut, Aymeric and Taylor, Adrien},
  title = {On fundamental proof structures in first-order optimization},
  booktitle = {2023 62nd {IEEE} Conference on Decision and Control ({CDC})},
  organization = {IEEE},
  pages = {3023--3030},
  year = {2023},
  doi = {10.1109/CDC49753.2023.10383282},
}

@article{drori2014performance,
  author = {Drori, Yoel and Teboulle, Marc},
  title = {Performance of first-order methods for smooth convex minimization: {A} novel approach},
  journal = {Mathematical Programming},
  volume = {145},
  number = {1},
  pages = {451--482},
  publisher = {Springer},
  year = {2014},
  doi = {10.1007/s10107-013-0653-0},
}

@article{taylor2017exact,
  author = {Taylor, Adrien B. and Hendrickx, Julien M. and Glineur, Fran{\c{c}}ois},
  title = {Exact worst-case performance of first-order methods for composite convex optimization},
  journal = {{SIAM} Journal on Optimization},
  volume = {27},
  number = {3},
  pages = {1283--1313},
  year = {2017},
  doi = {10.1137/16M108104X},
}

@article{taylor2017smooth,
  title     = {Smooth strongly convex interpolation and exact worst-case performance of first-order methods},
  author    = {Taylor, Adrien B. and Hendrickx, Julien M. and Glineur, Fran{\c{c}}ois},
  year      = 2017,
  journal   = {Mathematical Programming},
  publisher = {Springer},
  volume    = 161,
  number    = {1-2},
  pages     = {307--345},
  doi       = {10.1007/s10107-016-1009-3},
}

@article{taylor2018proximal,
  author = {Taylor, Adrien B. and Hendrickx, Julien M. and Glineur, Fran{\c{c}}ois},
  title = {Exact worst-case convergence rates of the proximal gradient method for composite convex minimization},
  journal = {Journal of Optimization Theory and Applications},
  volume = {178},
  number = {2},
  pages = {455--476},
  year = {2018},
  doi = {10.1007/s10957-018-1298-1},
}

@article{uschmajew2022note,
  author = {Uschmajew, Andr{\'e} and Vandereycken, Bart},
  title = {A Note on the Optimal Convergence Rate of Descent Methods with Fixed Step Sizes for Smooth Strongly Convex Functions},
  journal = {Journal of Optimization Theory and Applications},
  volume = {194},
  number = {1},
  pages = {364--373},
  year = {2022},
  doi = {10.1007/s10957-022-02032-z},
}

@inproceedings{fazel2003log,
  author = {Fazel, Maryam and Hindi, Haitham and Boyd, Stephen P.},
  title = {Log-det heuristic for matrix rank minimization with applications to {Hankel} and {Euclidean} distance matrices},
  booktitle = {Proceedings of the 2003 American Control Conference ({ACC}), 2003},
  organization = {IEEE},
  volume = {3},
  pages = {2156--2162},
  year = {2003},
  doi = {10.1109/ACC.2003.1243393},
}

@article{natarajan1995sparse,
  author = {Natarajan, B. K.},
  title = {Sparse approximate solutions to linear systems},
  journal = {{SIAM} Journal on Computing},
  volume = {24},
  number = {2},
  pages = {227--234},
  year = {1995},
  doi = {10.1137/S0097539792240406},
}

@article{bertsimas2016best,
  author = {Bertsimas, Dimitris and King, Angela and Mazumder, Rahul},
  title = {Best subset selection via a modern optimization lens},
  journal = {The Annals of Statistics},
  volume = {44},
  number = {2},
  pages = {813--852},
  year = {2016},
  doi = {10.1214/15-AOS1388},
}

@article{amaldi1998approximability,
  author = {Amaldi, Edoardo and Kann, Viggo},
  title = {On the approximability of minimizing nonzero variables or unsatisfied relations in linear systems},
  journal = {Theoretical Computer Science},
  volume = {209},
  number = {1--2},
  pages = {237--260},
  year = {1998},
  doi = {10.1016/S0304-3975(97)00115-1},
}

@article{tibshirani1996regression,
  author = {Tibshirani, Robert},
  title = {Regression shrinkage and selection via the lasso},
  journal = {Journal of the Royal Statistical Society: Series B (Methodological)},
  volume = {58},
  number = {1},
  pages = {267--288},
  year = {1996},
  doi = {10.1111/j.2517-6161.1996.tb02080.x},
}

@article{candes2008enhancing,
  author = {Cand{\`e}s, Emmanuel J. and Wakin, Michael B. and Boyd, Stephen P.},
  title = {Enhancing sparsity by reweighted {$\ell_1$} minimization},
  journal = {Journal of Fourier Analysis and Applications},
  volume = {14},
  number = {5},
  pages = {877--905},
  year = {2008},
  doi = {10.1007/s00041-008-9045-x},
}

@article{zhang2010analysis,
  author = {Zhang, Tong},
  title = {Analysis of multi-stage convex relaxation for sparse regularization},
  journal = {Journal of Machine Learning Research},
  volume = {11},
  number = {35},
  pages = {1081--1107},
  year = {2010},
  url = {https://jmlr.csail.mit.edu/papers/v11/zhang10a.html},
}

@misc{upadhyaya2025autolyap,
  author = {Upadhyaya, Manu and Das Gupta, Shuvomoy and Taylor, Adrien B. and Banert, Sebastian and Giselsson, Pontus},
  title = {The {AutoLyap} software suite for computer-assisted {Lyapunov} analyses of first-order methods},
  year = {2025},
  note = {\href{https://arxiv.org/abs/2506.24076}{arXiv:2506.24076 [math.OC]}},
}

@article{upadhyaya2024automated,
  author = {Upadhyaya, Manu and Banert, Sebastian and Taylor, Adrien B. and Giselsson, Pontus},
  title = {Automated tight {Lyapunov} analysis for first-order methods},
  journal = {Mathematical Programming},
  volume = {209},
  number = {1--2},
  pages = {133--170},
  year = {2024},
  doi = {10.1007/s10107-024-02061-8},
}

@misc{upadhyaya2026optimal,
  author = {Upadhyaya, Manu and Berg Thomsen, Daniel and Dieuleveut, Aymeric and Taylor, Adrien B.},
  title = {An optimal first-order method for smooth and strongly convex composite optimization and its stationary limit},
  year = {2026},
  note = {\href{https://arxiv.org/abs/2605.22929}{arXiv:2605.22929 [math.OC]}},
}

@inproceedings{taylor2017performance,
  title={Performance estimation toolbox ({PESTO}): {A}utomated worst-case analysis of first-order optimization methods},
  author={Taylor, Adrien B and Hendrickx, Julien M and Glineur, Fran{\c{c}}ois},
  booktitle={2017 IEEE 56th Annual Conference on Decision and Control (CDC)},
  pages={1278--1283},
  year={2017},
  organization={IEEE},
  doi={10.1109/CDC.2017.8263832},
}

@article{goujaud2024pepit,
  author = {Goujaud, Baptiste and Moucer, C{\'e}line and Glineur, Fran{\c{c}}ois and Hendrickx, Julien M and Taylor, Adrien B and Dieuleveut, Aymeric},
  title = {{PEPit}: {C}omputer-assisted worst-case analyses of first-order optimization methods in {Python}},
  journal = {Mathematical Programming Computation},
  volume = {16},
  number = {3},
  pages = {337--367},
  publisher = {Springer},
  year = {2024},
  doi = {10.1007/s12532-024-00259-7},
}

@article{kim2021accelerated,
  author = {Kim, Donghwan},
  title = {Accelerated proximal point method for maximally monotone operators},
  journal = {Mathematical Programming},
  volume = {190},
  pages = {57--87},
  year = {2021},
  doi = {10.1007/s10107-021-01643-0},
}

@article{guyang2023tight,
  author = {Guoyong Gu and Junfeng Yang},
  title = {Tight Convergence Rate in Subgradient Norm of the Proximal Point Algorithm},
  journal = {Optimization},
  year = {2025},
  doi = {10.1080/02331934.2025.2602877},
}

@article{ryu2020operator,
  title={Operator splitting performance estimation: {T}ight contraction factors and optimal parameter selection},
  author={Ryu, Ernest K and Taylor, Adrien B and Bergeling, Carolina and Giselsson, Pontus},
  journal={SIAM Journal on Optimization},
  volume={30},
  number={3},
  pages={2251--2271},
  year={2020},
  publisher={SIAM},
  doi={10.1137/19M1304854},
}

@article{dragomir2022optimal,
  title={Optimal complexity and certification of {Bregman} first-order methods},
  author={Dragomir, Radu-Alexandru and Taylor, Adrien B and d'Aspremont, Alexandre and Bolte, J{\'e}r{\^o}me},
  journal={Mathematical Programming},
  volume={194},
  number={1},
  pages={41--83},
  year={2022},
  publisher={Springer},
  doi={10.1007/s10107-021-01618-1},
}

@inproceedings{taylor2019stochastic,
  title={Stochastic first-order methods: {N}on-asymptotic and computer-aided analyses via potential functions},
  author={Taylor, Adrien and Bach, Francis},
  booktitle={Conference on Learning Theory},
  pages={2934--2992},
  year={2019},
  organization={PMLR},
  url={https://proceedings.mlr.press/v99/taylor19a.html},
}

@article{de2017worst,
  title={On the worst-case complexity of the gradient method with exact line search for smooth strongly convex functions},
  author={{de Klerk}, Etienne and Glineur, Fran{\c{c}}ois and Taylor, Adrien B},
  journal={Optimization Letters},
  volume={11},
  number={7},
  pages={1185--1199},
  year={2017},
  publisher={Springer},
  doi={10.1007/s11590-016-1087-4},
}

@article{drori2020efficient,
  title={Efficient First-order Methods for Convex Minimization: {A} constructive approach},
  author={Drori, Yoel and Taylor, Adrien},
  journal={Mathematical Programming, Series A},
  volume={184},
  pages={183--220},
  year={2020},
  doi={10.1007/s10107-019-01410-2},
}

@article{gorbunov2022last,
  title={Last-iterate convergence of optimistic gradient method for monotone variational inequalities},
  author={Gorbunov, Eduard and Taylor, Adrien and Gidel, Gauthier},
  journal={Advances in Neural Information Processing Systems},
  volume={35},
  pages={21858--21870},
  year={2022},
  doi={10.52202/068431-1589},
}

@misc{barre2020principled,
  title={Principled analyses and design of first-order methods with inexact proximal operators},
  author={Barr{\'e}, Mathieu and Taylor, Adrien and Bach, Francis},
  year={2020},
  note={\href{https://arxiv.org/abs/2006.06041}{arXiv:2006.06041 [math.OC]}},
}

@misc{goujaud2022optimal,
  title={Optimal first-order methods for convex functions with a quadratic upper bound},
  author={Goujaud, Baptiste and Taylor, Adrien and Dieuleveut, Aymeric},
  year={2022},
  note={\href{https://arxiv.org/abs/2205.15033}{arXiv:2205.15033 [math.OC]}},
}

@article{taylor2023optimal,
  title={An optimal gradient method for smooth strongly convex minimization},
  author={Taylor, Adrien and Drori, Yoel},
  journal={Mathematical Programming},
  volume={199},
  number={1},
  pages={557--594},
  year={2023},
  publisher={Springer},
  doi={10.1007/s10107-022-01839-y},
}

@article{bergthomsen2025tight,
  title={Tight analyses of first-order methods with error feedback},
  author={Berg Thomsen, Daniel and Taylor, Adrien and Dieuleveut, Aymeric},
  journal={Advances in Neural Information Processing Systems (NeurIPS)},
  year={2025},
  url={https://papers.neurips.cc/paper_files/paper/2025/hash/ea9f05e43ff75fc8d70ba9ff1d18cfa5-Abstract-Conference.html},
}

@article{kim2016optimized,
  title={Optimized first-order methods for smooth convex minimization},
  author={Kim, Donghwan and Fessler, Jeffrey A},
  journal={Mathematical Programming},
  volume={159},
  number={1},
  pages={81--107},
  year={2016},
  publisher={Springer},
  doi={10.1007/s10107-015-0949-3},
}

@article{kim2018generalizing,
  title={Generalizing the optimized gradient method for smooth convex minimization},
  author={Kim, Donghwan and Fessler, Jeffrey A},
  journal={SIAM Journal on Optimization},
  volume={28},
  number={2},
  pages={1920--1950},
  year={2018},
  publisher={SIAM},
  doi={10.1137/17M112124X},
}

@article{kim2018another,
  title={Another look at the fast iterative shrinkage/thresholding algorithm ({FISTA})},
  author={Kim, Donghwan and Fessler, Jeffrey A},
  journal={SIAM Journal on Optimization},
  volume={28},
  number={1},
  pages={223--250},
  year={2018},
  publisher={SIAM},
  doi={10.1137/16M108940X},
}

@article{lieder2021convergence,
  title={On the convergence rate of the {H}alpern-iteration},
  author={Lieder, Felix},
  journal={Optimization Letters},
  volume={15},
  number={2},
  pages={405--418},
  year={2021},
  publisher={Springer},
  doi={10.1007/s11590-020-01617-9},
}

@article{altschuler2025acceleration,
  title={Acceleration by stepsize hedging: {M}ulti-step descent and the silver stepsize schedule},
  author={Altschuler, Jason M and Parrilo, Pablo A},
  journal={Journal of the ACM},
  volume={72},
  number={2},
  pages={1--38},
  year={2025},
  publisher={ACM New York, NY},
  doi={10.1145/3708502},
}

@misc{shugart2025negative,
  title={Negative Stepsizes Make Gradient-Descent-Ascent Converge},
  author={Shugart, Henry and Altschuler, Jason M.},
  year={2025},
  note={\href{https://arxiv.org/abs/2505.01423}{arXiv:2505.01423 [math.OC]}},
}

@inproceedings{park2022exact,
  title={Exact Optimal Accelerated Complexity for Fixed-Point Iterations},
  author={Park, Jisun and Ryu, Ernest K.},
  booktitle={Proceedings of the 39th International Conference on Machine Learning},
  pages={17420--17457},
  year={2022},
  volume={162},
  series={Proceedings of Machine Learning Research},
  publisher={PMLR},
  url={https://proceedings.mlr.press/v162/park22c.html},
}

@inproceedings{yoon2024optimal,
  title={Optimal Acceleration for Minimax and Fixed-Point Problems is Not Unique},
  author={Yoon, Taeho and Kim, Jaeyeon and Suh, Jaewook J. and Ryu, Ernest K.},
  booktitle={Proceedings of the 41st International Conference on Machine Learning},
  pages={57244--57314},
  year={2024},
  volume={235},
  series={Proceedings of Machine Learning Research},
  publisher={PMLR},
  url={https://proceedings.mlr.press/v235/yoon24b.html},
}

@misc{rotaru2024exact,
  title={Exact worst-case convergence rates of gradient descent: {A} complete analysis for all constant stepsizes over nonconvex and convex functions},
  author={Rotaru, Teodor and Glineur, Fran{\c{c}}ois and Patrinos, Panagiotis},
  year={2024},
  note={\href{https://arxiv.org/abs/2406.17506}{arXiv:2406.17506 [math.OC]}},
}

@phdthesis{rotaru2026thesis,
  title = {Exact Performance Analysis of Fundamental First-Order Optimization Methods},
  author = {Rotaru, Teodor},
  school = {KU Leuven},
  year = {2026},
  url = {https://lirias.kuleuven.be/handle/20.500.12942/781389},
}

@inproceedings{naldi2025solving,
  title={Solving generic parametric linear matrix inequalities},
  author={Naldi, Simone and Safey El Din, Mohab and Taylor, Adrien and Wang, Weijia},
  booktitle={Proceedings of the 2025 International Symposium on Symbolic and Algebraic Computation},
  pages={267--276},
  year={2025},
  doi={10.1145/3747199.3747570},
}

@article{abbaszadehpeivasti2024rate,
  title={On the rate of convergence of the difference-of-convex algorithm ({DCA})},
  author={Abbaszadehpeivasti, Hadi and {de Klerk}, Etienne and Zamani, Moslem},
  journal={Journal of Optimization Theory and Applications},
  volume={202},
  number={1},
  pages={475--496},
  year={2024},
  publisher={Springer},
  doi={10.1007/s10957-023-02199-z},
}

@article{jang2025computer,
  title={Computer-assisted design of accelerated composite optimization methods: {OptISTA}},
  author={Jang, Uijeong and Das Gupta, Shuvomoy and Ryu, Ernest K},
  journal={Mathematical Programming},
  pages={1--109},
  year={2025},
  publisher={Springer},
  doi={10.1007/s10107-025-02258-5},
}
